\documentclass[12pt]{article}%
\usepackage{amsmath}
\usepackage{amsfonts}
\usepackage{mitpress}
\usepackage{amssymb}
\usepackage{graphicx}
\usepackage{epsf,exscale,times}
\usepackage[english]{babel}
\usepackage{dsfont}
\usepackage{geometry}
\usepackage{subfig}%
\providecommand{\U}[1]{\protect\rule{.1in}{.1in}}
\numberwithin{equation}{section}

\def\eqlaw{\buildrel d \over =}

\newtheorem{theorem}{Theorem}[section]

\newtheorem{corollary}{Corollary}[section]

\newtheorem{definition}{Definition}[section]
\newtheorem{example}{Example}[section]

\newtheorem{lemma}{Lemma}[section]

\newtheorem{proposition}{Proposition}[section]
\newtheorem{remark}{Remark}[section]

\newenvironment{proof}[1][Proof]{\noindent\textbf{#1.} }{\ \rule{0.5em}{0.5em}}
\newdimen\dummy
\dummy=\oddsidemargin
   \graphicspath{{../images/}}

\begin{document}

\title{Estimation for models defined by conditions on their L-moments}
\author{Michel Broniatowski$^{(1)}$ , Alexis Decurninge $^{(1,\ast)}$}
\maketitle

\begin{abstract}
This paper extends the empirical minimum divergence approach for models which
satisfy linear constraints with respect to the probability measure of the
underlying variable (moment constraints) to the case where such constraints
pertain to its quantile measure (called here semi parametric quantile
models).\ The case when these constraints describe shape conditions as handled
by the L-moments is considered and both the description of these models as
well as the resulting non classical minimum divergence procedures are
presented. These models describe neighborhoods of classical models used mainly
for their tail behavior, for example neighborhoods of Pareto or Weibull
distributions, with which they may share the same first L-moments. A parallel
is drawn with similar problems held in optimal transportation problems. The
properties of the resulting estimators are illustrated by simulated examples
comparing Maximum Likelihood estimators on Pareto and Weibull models to the
minimum Chi-square empirical divergence approach on semi parametric quantile
models, and others.

$^{(1)}$LSTA, Universit\'{e} Pierre et Marie Curie, Paris, France

$^{(\ast)}$Corresponding author.

\end{abstract}
\tableofcontents

\section{\bigskip Motivation and notation}

\label{section1} For univariate distributions, L-moments are expressed as the
expectation of a particular linear combination of order statistics. Let us
consider $r$ independent copies $X_{1},...,X_{r}$ of a random variable $X$
with $\mathbb{E}\left(  \left\vert X\right\vert \right)  $ a finite number.
The $r$-th L-moment is defined by
\begin{equation}
\lambda_{r}=\frac{1}{r}\sum_{k=0}^{r-1}(-1)^{k}\dbinom{r-1}{k}\mathbb{E}%
[X_{r-k:r}] \label{L-moments}%
\end{equation}
where $X_{1:r}\leq...\leq X_{r:r}$ denotes the order statistics. The four
first L-moment can be considered as a measure of location, dispersion,
skewness and kurtosis. Indeed $\lambda_{1}=\mathbb{E}(X)$, $\lambda_{2}$ is
expressed as $\lambda_{2}=\left(  1/2\right)  \mathbb{E}\left(  \left\vert
X-Y\right\vert \right)  $ with $Y$ an independent copy of $X$, $\lambda_{3}$
indicates the expected distance between the mean of the extreme terms and the
median one in a sample of three i.i.d. replications of $X$, and $\lambda_{4}$
is an indicator of the expected distance between the extreme terms of a sample
of four replicates of $X$ with respect to a multiple of the distance between
the two central terms.

L-moments constitute a robust alternative to traditional moments as
descriptors of a distribution since only the existence of $\mathbb{E}\left(
\left\vert X\right\vert \right)  $ is needed in order to insure their
existence. Since their introduction in Hosking's paper in 1990
(\cite{hosking90}), methods based on L-moments have become popular especially
in applications dealing with heavy-tailed distributions. As mentioned in
\cite{hosking90} and \cite{hosking92}:"The main advantage of L-moments over
conventional moments is that L-moments, being linear functions of the data,
suffer less from the effect of sampling variability: L-moments are more robust
than conventional moments to outliers in the data and enable more secure
inferences to be made from small samples about an underlying probability
distribution. Also as seen through (\ref{L-moments}) the L-moments are
determined by the expectation of extreme order statistics, and vice versa".
This motivates their success for the inference in models pertaining to the
tail behavior of random phenomenons.

In this article, we will consider semi-parametric models conditioned by
constraints on a finite number of L-moments. Let us mention three examples of
such models; the two first examples describe neighborhoods of the Weibull and
the Pareto models, which are classical benchmarks for the description of tail
properties, and the third one describes a family of distributions which
express some loose symmetry property.

\begin{example}
\label{example1}We first consider the model which is the family of all the
distributions of a r.v. $X$ whose second, third and fourth L-moments verify :
\begin{equation}
\left\{
\begin{array}
[c]{lll}%
\lambda_{2}=\sigma(1-2^{-1/\nu})\Gamma(1+1/\nu) &  & \\
\lambda_{3}=\lambda_{2}[3-2\frac{1-3^{-1/\nu}}{1-2^{-1/\nu}})] &  & \\
\lambda_{4}=\lambda_{2}[6+\frac{5(1-4^{-1/\nu})-10(1-3^{-1/\nu})}{1-2^{-1/\nu
}})] &  &
\end{array}
\right.  \label{ModeleWeibull}%
\end{equation}
for any $\sigma>0,\nu>0$. These distributions share their first L-moments\ of
order 2, 3 and 4 with those of a Weibull distribution with scale and shape
parameter $\sigma$ and $\nu$. When $X$ is substituted by $Y:=X+a$ for some
real number $a$ then the distribution of $Y$ is Weibull with a shifted
support, hence with the same parameters $\sigma$ and $\nu$ as $X$; the r.v.
$Y$ shares the same L-moments $\lambda_{r}$ with those of $X$ but for $r=1$
and the model (\ref{ModeleWeibull}) describes a neighborhood of the continuum
of all Weibull distributions on $\left[  a,\infty\right)  $ or on $\left(
-\infty,a\right]  $ when $a$ belongs to $\mathbb{R}.$ Hence this model aims at
describing a shape constraint on the tail of the distribution of the data,
independently of its location.\ 
\end{example}

\begin{example}
\label{example2}Secondly, we consider the model which is the space of the
distributions whose second, third and fourth L-moments verify :
\begin{equation}
\left\{
\begin{array}
[c]{lll}%
\lambda_{2} & = & \frac{\sigma}{(1-\nu)(2-\nu)}\\
\lambda_{3} & = & \lambda_{2}\frac{1+\nu}{3-\nu}\\
\lambda_{4} & = & \lambda_{2}\frac{(1+\nu)(2+\nu)}{(3-\nu)(4-\nu)}%
\end{array}
\right.  \label{Modele pareto}%
\end{equation}
for any $\sigma>0,\nu\in\mathbb{R}$. These distributions share their first
L-moments with those of a generalized Pareto distribution with scale and shape
parameter $\sigma$ and $\nu$. The same remark as in the above example holds;
model (\ref{Modele pareto}) describes a neighborhood of the whole continuum of
Pareto distributions on $\left[  a,\infty\right)  $ or on $\left(
-\infty,a\right]  $ when $a$ belongs to $\mathbb{R}.$
\end{example}

\begin{example}
\label{example3}Let finally be given an appealing example based on order
statistics, namely
\[
\left\{
\begin{array}
[c]{lll}%
\mathbb{E}[X_{1:3}]=\theta-\nu &  & \\
\mathbb{E}[X_{2:3}]=\theta &  & \\
\mathbb{E}[X_{3:3}]=\theta+\nu &  &
\end{array}
\right.
\]
for any $\theta\in\mathbb{R},\nu>0$.
\end{example}

Before any further discussion on the scope of the present paper, a few
notation seems useful. For a non decreasing function $F$ with bounded
variation on any interval of $\mathbb{R}$ we denote $\mathbf{F}$ the
corresponding positive $\sigma-$finite measure on $\left(  \mathbb{R}%
,\mathcal{B}\left(  \mathbb{R}\right)  \right)  .$ For example when $F$ is the
distribution function of a probability measure, then this measure is denoted
$\mathbf{F}$ or $dF\mathfrak{.}$ Denote in this case
\[
F^{-1}(u):=\inf\left\{  x\in\mathbb{R}\text{ s.t. }F(x)\geq u\right\}  \text{
for }u\in\left(  0,1\right)
\]
the generalized inverse of $F$ , a left continuous non decreasing function
which is the quantile function of the probability measure $\mathbf{F}%
\mathfrak{.}$ Denote accordingly $\mathbf{F}^{-1}$ or $dF^{-1},$
indifferently, the quantile measure with distribution function $F^{-1}.$ If
$x_{1},.$ $.,x_{n}$ are $n$ realizations of a random variable $X$ with
absolutely continuous probability measure $\mathbf{F}$ then the gaps in the
empirical distribution function
\[
F_{n}(x):=\frac{1}{n}%
{\displaystyle\sum\limits_{i=1}^{n}}
{\Large 1}_{\left(  -\infty,x\right]  }\left(  x_{i}\right)
\]
are of size $1/n$ and are located on the $X_{i}$'s; the empirical quantile
function satisfies
\[
F_{n}^{-1}(u)=x_{i:n}\text{ when }\frac{i-1}{n}<u\leq\frac{i}{n}%
\]
and its gaps are given by
\[
F_{n}^{-1}\left(  \left(  i/n\right)  ^{+}\right)  -F_{n}^{-1}\left(  \left(
i/n\right)  \right)  =\mathbf{F}_{n}^{-1}\left(  i/n\right)  =x_{i+1:n}%
-x_{i:n}%
\]
where $x_{1:n}\leq...\leq x_{n:n}$ denotes the ordered sample; those gaps will
be denoted $\mathbf{F}_{n}^{-1}\left(  i/n\right)  $ or $dF_{n}^{-1}\left(
i/n\right)  $ indifferently; the empirical quantile measure has as its support
the uniformely sparsed points $\left\{  1/n,2/n,..,1\right\}  $ and attributes
masses equal to sampled spacings at those points; it follows that the
empirical quantile measure is a positive finite measure with finite support.
The quantile measure associated with the distribution function $F^{-1}$ is
also a positive $\sigma-$finite measure, defined on $\left(  0,1\right)  .$
The above construction defined the quantile measure from the probability
measure, but the reciprocal construction will be used, starting from a
quantile measure, defining its distribution function, turning to its inverse
to define a distribution of a probability measure, and then to the probability
measure itself.

We now turn back to our topics.

Models defined as in the above examples extend the classical parametric ones,
and are defined through some constraints on the form of the distributions.
They can be paralleled with models defined through moments conditions defined
as follows.

Let $\theta$ in $\Theta$ , an open subset of $\mathbb{R}^{d}$ and let
$g:(x,\theta)\in\mathbb{R}\times\Theta\rightarrow\mathbb{R}^{l}$ be a
$l$-valued function, each component of which is parametrized by $\theta
\in\Theta\subset\mathbb{R}^{d}$. Define
\[
M_{\theta}:=\left\{  \mathbf{F}\text{ s.t. }\int_{\mathbb{R}}g(x,\theta
)\mathbf{F}(dx)=0\right\}
\]
and the semi parametric model defined by moment conditions is the collection
of probability measures in
\begin{equation}
\mathcal{M}:=\bigcup_{\theta\in\Theta}M_{\theta}. \label{mom_model}%
\end{equation}
These semiparametric models are defined by $l$ conditions pertaining to $l$
moments of the distributions and are widely used in applied statistics. When
the dimension $d$ of the parameter space exceeds $l$, no plug-in method can
achieve any inference on $\theta$; however, various techniques have been
proposed in this case; see for example Hansen \cite{hansen96}, who defined the
Generalized Method of Moments (GMM) and Owen, who defined the so-called
empirical likelihood approach \cite{owen90}. Later, Newey and Smith
\cite{neweysmith04} or Broniatowski and Keziou \cite{broniatowski12} proposed
a refinement of the GMM approach minimizing a divergence criterion over the
model. A major feature of models defined by (\ref{mom_model}) lies in their
linearity with respect to the cumulative distribution function (cdf) which
brings a dual formulation of the minimization problem. Duality results easily
lead to the consistency and the asymptotic normality of the estimators of
$\theta$; see \cite{broniatowski12}\cite{neweysmith04}.\newline

Similarly as for models defined by (\ref{mom_model}), we can introduce
semiparametric linear quantile (SPLQ) models through
\begin{equation}
\bigcup_{\theta\in\Theta}L_{\theta}:=\bigcup_{\theta\in\Theta}\left\{
\mathbf{F}\text{ s.t. }\int_{0}^{1}F^{-1}(u)k(u,\theta)du=f(\theta)\right\}
\label{lmom_model}%
\end{equation}
where $\Theta\subset\mathbb{R}^{d}$, $k:(u,\theta)\in\lbrack0;1]\times
\Theta\rightarrow\mathbb{R}^{l}$ and $f:\Theta\rightarrow\mathbb{R}^{l}$. In
the above display, in accordance with the above notation, $F^{-1}$ denotes the
generalized inverse function of $F$, the distribution function of the measure
$\mathbf{F}$. Examples \ref{example1},\ref{example2} and \ref{example3} can be
written through (\ref{lmom_model}); see Section \ref{models_lmom}. We will
consider the case when $k$ is a function of $u$ only; this class contains many
examples, typically models defined by a finite number of constraints on
functions of the moments of the order statistics.\newline It is natural to
propose similar estimation procedures for SPLQ models based on a minimization
of a divergence. Models (\ref{lmom_model}) do not enjoy linearity with respect
to the cdf but with respect to the quantile function. Thus, as developed for
models defined by (\ref{mom_model}), we propose to minimize a divergence
criterion built on quantiles.\newline We will reformulate this criterion into
a minimization of the energy of a deformation of the empirical distribution. A
duality result and the subsequent consistency and asymptotic normality for the
corresponding family of estimators are presented in Sections \ref{section5}
and \ref{section7}.\newline Section \ref{deformation} draws a parallel with
with an optimal transportation approach.

In the following, the transpose of a vector $A$ will be denoted $A^{T}$ and if
$F$ and $G$ are two cdf's, $F\ll G$ means that $F$ is absolutely continuous
with respect to $G$. The Lebesgue measure on $\mathbb{R}$ is denoted
$d\lambda$ or $dx$, according to the common use in the context.


\section{L-moments}

\subsection{Definition and characterizations}

Let us consider data consisting in $\underline{X}=(x_{1},...,x_{r})$, which
are $r$ realizations of real-valued independent and identically distributed
(iid) copies $X_{1},..,X_{r}$ of a random variable (r.v.) $X$ with
distribution function $F$. The $r$-th L-moment $\lambda_{r}$ is defined by
\begin{equation}
\lambda_{r}=\frac{1}{r}\sum_{k=0}^{r-1}(-1)^{k}\dbinom{r-1}{k}\mathbb{E}%
[X_{r-k:r}] \label{eq:L_mom}%
\end{equation}
where $X_{1:r}\leq X_{2:r}\leq...\leq X_{r:r}$ denotes the order statistics of
$X_{1},..,X_{r}$.\newline From the above definition all L-moments $\lambda
_{r}$ but $\lambda_{1}$ are shift invariant, hence independent upon
$\lambda_{1}$. If $F$ $\ $is continuous, the expectation of the $j$-th order
statistics $X_{j:r}$ is (see David p.33\cite{david81})
\begin{equation}
\mathbb{E}[X_{j:r}]=\frac{r!}{(j-1)!(r-j)!}\int_{\mathbb{R}}xF(x)^{j-1}%
(1-F(x))^{r-j}\mathbf{F}(dx). \label{eq:lmom}%
\end{equation}
The first four L-moments are
\[%
\begin{array}
[c]{lll}%
\lambda_{1} & = & \mathbb{E}[X]\\
\lambda_{2} & = & \frac{1}{2}\mathbb{E}[X_{2:2}-X_{1:2}]\\
\lambda_{3} & = & \frac{1}{3}\mathbb{E}[X_{3:3}-2X_{2:3}+X_{1:3}]\\
\lambda_{4} & = & \frac{1}{4}\mathbb{E}[X_{4:4}-3X_{3:4}+3X_{2:4}-X_{1:4}].
\end{array}
\]

\begin{remark}
The second L-moment is equal to the half of the absolute mean difference
\[
\lambda_{2}=\frac{1}{2}\mathbb{E}[|X-Y|]
\]
where $X$ and $Y$ are independently sampled from the same distribution $F$.
The ratio $\frac{\lambda_{2}}{\lambda_{1}}$ is known as the Gini coefficient.
\end{remark}

The expectations of the extreme order statistics characterize a distribution:
if $\mathbb{E}\left(  \left\vert X\right\vert \right)  $ is finite, either of
the sets $\left\{  \mathbb{E}\left(  X_{1:n}\right)  ,n=1,..\right\}  $ or
$\left\{  \mathbb{E}\left(  X_{n:n}\right)  ,n=1,..\right\}  $ characterize
the distribution of $X;$ see \cite{Chan67} and \cite{Konheim71}. Since the
moments of order statistics are defined by the family of L-moments, those also
characterize the distribution of $X.$

The $r-$th L-moment ratio is defined for $r\geq2$ by
\[
\tau_{r}=\frac{\lambda_{r}}{\lambda_{2}}.
\]
The interpretation of $\lambda_{1},\lambda_{2},\tau_{3},\tau_{4}$ as measures
of location, scale, skewness and kurtosis respectively and the existence of
all L-moments whenever $\int|x|\mathbf{F}(dx)<\infty$ makes them good
alternatives to moments.

\begin{remark}
\label{RemStieltjes}We can define from the quantile function $F^{-1}%
:[0;1]\rightarrow\mathbb{R}$ an associated measure on $\mathcal{B}([0;1])$
\[
\mathbf{F}^{-1}(B)=\int_{0}^{1}\mathds{1}_{x\in B}dF^{-1}(x)\in\mathbb{R}%
\cup\{-\infty,+\infty\}.
\]
The above integral is a Riemann-Stieltjes integral. It defines a $\sigma
$-finite measure since $F^{-1}$ has bounded variations on every interval of
the form $[a,b]$ with $0<a\leq b<1$. For any $\mathbf{F}^{-1}$-measurable
function $a:\mathbb{R}\rightarrow\mathbb{R}$ , it holds
\[
\int_{0}^{1}a(x)dF^{-1}(x)=\int_{0}^{1}a(x)\mathbf{F}^{-1}(dx).
\]

\end{remark}

Writing the L-moments of a distribution $F$ as an inner product of the
corresponding quantile function with a specific complete orthogonal system of
polynomials in $L^{2}\left(  0,1\right)  $ is a cornerstone in the derivation
of statistical inference in SPLQ models. The shifted Legendre polynomials
define such a system of functions.

\begin{definition}
\label{DefLegendreshift}The shifted Legendre polynomial of order $r$ is
\begin{equation}
L_{r}(t)=\sum_{k=0}^{r}(-1)^{k}\dbinom{r}{k}^{2}t^{r-k}(1-t)^{k}=\sum
_{k=0}^{r}(-1)^{r-k}\dbinom{r}{k}\dbinom{r+k}{k}t^{k}.
\label{ShiftedLegendre-L_r}%
\end{equation}

For $r\geq1$ define $K_{r}$ as the integrated shifted Legendre polynomials
\begin{equation}
K_{r}(t)=\int_{0}^{t}L_{r-1}(u)du=-t(1-t)\frac{J_{r-2}^{(1,1)}(2t-1)}{r-1}
\label{jacobi}%
\end{equation}
with $J_{r-2}^{(1,1)}$ the corresponding Jacobi polynomial (see
\cite{hosking89})
\[
J_{r-2}^{(1,1)}(2t-1)=\frac{\Gamma(r)}{(r-2)!\Gamma(r+1)}\sum_{k=0}%
^{r-2}\dbinom{k}{r-2}\frac{\Gamma(r+1+k)}{\Gamma(2+k)}(t-1)^{k}.
\]

\end{definition}

The following result holds.

\begin{proposition}
\label{Prop Tonelli}Let $F$ be any cdf and assume that $\int\left\vert
x\right\vert dF(x)$ is finite. Then for any $r\geq1,$ it holds
\begin{equation}
\lambda_{r}=\int_{0}^{1}F^{-1}(t)L_{r-1}(t)dt=\int_{0}^{1}F^{-1}(t)dK_{r}(t)
\label{eq:lmom1}%
\end{equation}
where the last integral is the Stieltjes integral of $F^{-1}$ with respect to
the function $t\mapsto K_{r}(t).$
\end{proposition}

\begin{proof}
The proof is based on the following fundamental Lemma, whose proof is deferred
to the Appendix.

\begin{lemma}
\label{transport1}Let $U$ be a uniform random variable on [0;1] and $X$ be a
random variable with $F$. Then $F^{-1}(U)=_{d}X$.
\end{lemma}

Let $U_{1},...,U_{r}$ be $r$ independent random variable uniformly distributed
on $[0;1]$ and denote by $U_{1:r}\leq...\leq U_{r:r}$ the ordered statistics.
Then
\[
(X_{1:r},...,X_{r:r})\eqlaw (F^{-1}(U_{1:r}),...,F^{-1}(U_{r:r})); %
\]
hence for $1\leq j\leq r$
\[
\mathbb{E}[X_{j:r}]=\mathbb{E}[F^{-1}(U_{j:r})]=\frac{r!}{(j-1)!(r-j)!}%
\int_{0}^{1}F^{-1}(t)t^{j-1}(1-t)^{r-j}dt,
\]
which ends the proof of Proposition \ref{Prop Tonelli}.
\end{proof}

Before going any further, we present an useful Lemma, the proof of which is
also deferred to the Appendix.

\begin{lemma}
\label{stieljes}Let a be a real-valued function such that $\int_{\mathbb{R}%
}a(x)dF(x)<\infty$. Then
\begin{equation}
\int_{\mathbb{R}}a(x)d\mathbf{F}(x)=\int_{0}^{1}a(F^{-1}(t))dt.
\label{eq_stieltjes0}%
\end{equation}
Similarly if $t\rightarrow b(t)$ is a real-valued function such that $\int
_{0}^{1}b(t)\mathbf{F}^{-1}(dt)<\infty$. Then
\begin{equation}
\int_{0}^{1}b(t)\mathbf{F}^{-1}(dt)=\int_{0}^{1}b(F(x))dx.
\label{eq_stieltjes}%
\end{equation}

\end{lemma}

\begin{remark}
\label{remark2} As a consequence of Lemma \ref{stieljes} and equation
(\ref{eq:lmom1}), it holds
\[
\lambda_{r}=\int_{0}^{1}xdK_{r}(F(x)).
\]

\end{remark}

\begin{remark}
If we consider a multinomial distribution with support $x_{1}\leq x_{2}
\leq... \leq x_{n}$ and associated weights $\pi_{1},...,\pi_{n}$ ($\sum
_{i=1}^{n} \pi_{i} = 1$), we get
\[
\lambda_{r} = \sum_{i=1}^{n} w_{i}^{(r)} x_{i} = \sum_{i=1}^{n} \left[  K_{r}
\left(  \sum_{a=1}^{i} \pi_{a} \right)  - K_{r}\left(  \sum_{a=1}^{i-1}
\pi_{a} \right)  \right]  x_{i} = \int_{0}^{1} L_{r-1}(t)Q_{\pi}(t)dt
\]
with
\[
Q_{\pi}(t) = \left\{
\begin{array}
[c]{ll}%
x_{1} & \text{ if \ \ $0\leq t\leq\pi_{1}$}\\
x_{i} & \text{ if \ \ $\sum_{a=1}^{i-1} \pi_{a}< t \leq\sum_{a=1}^{i} \pi_{a}%
$}%
\end{array}
\right.  .
\]
This example illustrates Remark \ref{remark2}.%

\begin{figure}
[ptb]
\begin{center}
\includegraphics[
natheight=9.375400in,
natwidth=13.333700in,
height=2.4102in,
width=3.4169in
]%
{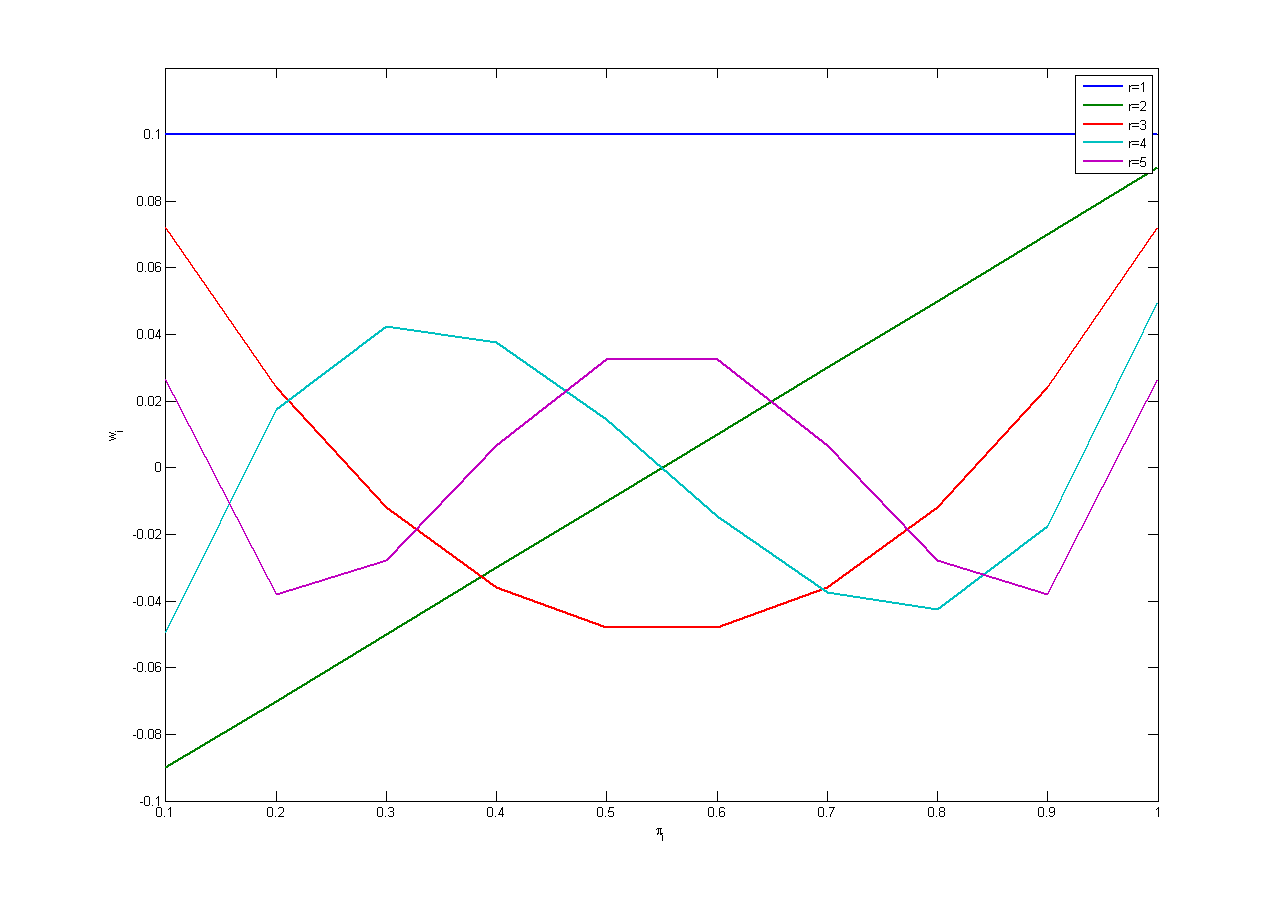}%
\end{center}
\caption{Weights $w_i^{(r)}$ for the uniform law with a support containing 10 points}
\label{fig:poids_lmm}
\end{figure}
Figure \ref{fig:poids_lmm} provides the first weight $w_{i}^{(r)}$ when the
$x_{i}$'s are equally sparsed on $\left[  0,1\right]  $ with equal weights
$\pi_{1}=..=\pi_{n}=1/n.$
\end{remark}

The following characterization for the L-moments with order larger or equal to
$2$ is used in Section \ref{models_lmom}.

\begin{proposition}
\label{prop02}If $r\geq2$ and $\int_{\mathbb{R}}\left\vert x\right\vert
dF(x)<+\infty$, then
\begin{equation}
\lambda_{r}=\int_{0}^{1}F^{-1}(t)dK_{r}(t)=-\int_{0}^{1}K_{r}(t)\mathbf{F}%
^{-1}(dt). \label{eq:lmom2}%
\end{equation}

\end{proposition}

\begin{proof}
This result follows as an application of Fubini-Tonelli Theorem. Indeed
\begin{align*}
\lambda_{r}  &  =\int_{0}^{1}F^{-1}(t)dK_{r}(t)\\
&  =\int_{0}^{1}\int_{0}^{t}\mathbf{F}^{-1}(du)dK_{r}(t)\\
&  =\int_{0}^{1}\int_{0}^{1}\mathds{1}_{0\leq u\leq t}\mathbf{F}%
^{-1}(du)dK_{r}(t).
\end{align*}
This last equality holds since $(u,t)\mapsto\mathds{1}_{0\leq u\leq t}$ is
measurable with respect to the measure $\mathbf{F}^{-1}\times dK_{r}$ since
$\mathbb{E}[X]<\infty$. Applying Fubini-Tonelli Theorem, it holds \newline%
\begin{align*}
\lambda_{r}  &  =\int_{0}^{1}\int_{0}^{1}\mathds{1}_{0\leq u\leq t}%
dK_{r}(t)\mathbf{F}^{-1}(du)\\
&  =\int_{0}^{1}\int_{0}^{1}\left[  K_{r}(1)-K_{r}(u)\right]  \mathbf{F}%
^{-1}(du)\\
&  =-\int_{0}^{1}K_{r}(u)\mathbf{F}^{-1}(du)
\end{align*}
since $K_{r}(1)=0$ for $r>1$.
\end{proof}

\begin{remark}
\label{remark_loc} That (\ref{eq:lmom2}) does not hold for $r=1$ follows from
the fact that if $G=F(.+a)$ for some $a\in\mathbb{R}$, then $\mathbf{G}%
^{-1}=\mathbf{F}^{-1}$. Hence, SPLQ models are shift-invariant. This can also
be seen setting $r=1$ in the right-hand side of (\ref{eq:lmom2}); in this
case, the integral is infinite (but if $supp(\mathbf{F})$ is bounded) whereas
$\lambda_{1}$ is supposed to be finite.
\end{remark}

\subsection{Estimation of L-moments}

Let $x_{1},...,x_{n}$ be iid realizations of a random variable $X$ with
distribution $F$ and L-moments $\lambda_{r}$. Define $F_{n}$ the empirical cdf
of the sample and $l_{r}$ the corresponding plug-in estimator of $\lambda_{r}%
$,
\begin{equation}
l_{r}=\int_{0}^{1}F_{n}^{-1}(t)L_{r-1}(t)dt. \label{empirical_vstat}%
\end{equation}
This estimator of $\lambda_{r}$ is biased as quoted in \cite{hosking90} and
\cite{xiao07}. $l_{r}$ is usually termed as a V-statistic. As noted upon in
\cite{hosking90} and \cite{xiao07}, the unbiased estimators of L-moments are
the following U-statistics
\[
l_{r}^{(u)}=\frac{1}{\dbinom{n}{r}}\sum_{1\leq i_{1}<\dots<i_{r}\leq n}%
\frac{1}{r}\sum_{k=0}^{r-1}(-1)^{k}\dbinom{r-1}{k}x_{i_{r-k}:n}.
\]

\begin{remark}
An alternative definition for $l_{r}$ as in (\ref{empirical_vstat}) can be
stated as follows. Conditionally on the realizations $x=(x_{1},...,x_{n})$,
define the uniform distribution on $x$. Then $l_{r}$ is the discrete L-moment
of order $r$ of this conditional distribution. It can therefore be defined
through
\[
l_{r}=\frac{1}{\dbinom{r+n-1}{n-1}}\sum_{1\leq i_{1}\leq\dots\leq i_{r}\leq
n}\frac{1}{r}\sum_{k=0}^{r-1}(-1)^{k}\dbinom{r-1}{k}x_{i_{r-k}:n}.
\]
Let us now extend Definition \ref{eq:L_mom} of the L-moments as follows. Let
$(i_{1},...,i_{r})$ be drawn without replacement from $\{1,...,r\}$. We then
define $x_{(i_{1})}\leq...\leq x_{(i_{r})}$ the corresponding ordered
observations and
\[
\lambda_{r}^{(u)}=\frac{1}{r}\sum_{k=0}^{r-1}(-1)^{k}\dbinom{r-1}{k}%
\mathbb{E}[x_{(i_{r-k})}]
\]
where the expectation is taken under the extraction process. Then $\lambda
_{r}^{(u)}$ and $l_{r}^{(u)}$ coincide.\newline Although $l_{r}^{(u)}$ is
unbiased, for sake of simplicity only $l_{r}$ which is asymptotically
unbiased, will be used in the sequel.
\end{remark}

These two estimators $l_{r}$ and $l_{r}^{(u)}$ of the L-moment $\lambda_{r}$
have the same asymptotic properties.

\begin{proposition}
Let us suppose that $F$ has finite variance. Then, for any $m\geq1$%
\[
\sqrt{n}\left[  \left(
\begin{array}
[c]{c}%
l_{1}\\
\vdots\\
l_{m}%
\end{array}
\right)  -\left(
\begin{array}
[c]{c}%
\lambda_{1}\\
\vdots\\
\lambda_{m}%
\end{array}
\right)  \right]  \rightarrow_{d}\mathcal{N}_{m}(0,\Lambda)
\]
where $\mathcal{N}_{m}$ denotes the multivariate normal distribution and the
elements of $\Lambda$ are given by
\[
\Lambda_{rs}=\int\int_{x<y}\left[  L_{r-1}(F(x))L_{s-1}(F(y))+L_{r-1}%
(F(y))L_{s-1}(F(x))\right]  F(x)(1-F(y))dxdy
\]
Furthermore, the same property holds for $l_{1},..,l_{r}$ substituted by
$l_{1}^{(u)},...,l_{m}^{(u)}$.
\end{proposition}

\begin{proof}
This is a plain consequence of Theorem 6 in \cite{stigler74}. See also
\cite{hosking90} for an evaluation of the bias of $l_{r}$.
\end{proof}


\section{Models defined by moment and L-moment equations}

\subsection{Models defined by moment conditions}

Let us consider $n$ iid random variables $X_{1}$,...,$X_{n}$ drawn from the
same distribution function $F$. Semi-parametric models are often defined
through equations :
\[
\int_{\mathbb{R}}g(x,\theta)\mathbf{F}(dx)=\mathbb{E}[g(X,\theta)]=0
\]
where $g:\mathbb{R}\times\Theta\rightarrow\mathbb{R}^{l}$ and $\Theta
\subset\mathbb{R}^{d}$ is a space of parameters, as quoted in Section
\ref{section1}.\newline

\begin{example}
We can sometimes face distributions with constraints pertaining to the two
first moments. For example, Godambe and Thompson \cite{godambe89} considered
the distributions verifying $\mathbb{E}[X]=\theta$ and $\mathbb{E}%
[X^{2}]=h(\theta)$ with a known function $h$. Then, with our notations $l=2$
and $g(x,\theta)=(x-\theta,x^{2}-h(\theta))$
\end{example}

\begin{example}
Consider the distributions $F$ such that for some $\theta$ it holds
$F(y)=1-F(-y)=\theta$ \cite{broniatowski12}. This corresponds to a moment
condition model with $l=2$ and $g(x,\theta)=(\mathds{1}_{]-\infty
;y]}(x)-\theta,\mathds{1}_{[y;+\infty\lbrack}(x)-\theta)$. The condition on
the model is the existence of some $\theta$ such that the left and right
quantiles of order $\theta$ are $-y$ and $+y$ for some given $y$.
\end{example}


\subsection{Models defined by L-moments conditions}

\label{models_lmom} In the present paper we consider models defined by $l$
constraints on their first L-moments, namely satisfying
\begin{equation}
-\mathbb{E}\left[  \frac{1}{r}\sum_{k=0}^{r-1}(-1)^{k}\dbinom{r-1}{k}%
X_{k:r}\right]  =f_{r}(\theta)\ \ \ 1\leq r\leq l \label{eq_lmm}%
\end{equation}
where $\Theta$ is some open set in $\mathbb{R}^{d}$ and $f_{r}:\Theta
\rightarrow\mathbb{R}$ are some given functions defined on $\Theta$ , $1\leq
r\leq l.$\newline Those models are SPLQ, with $(u,\theta)\mapsto k(u,\theta)$
independent on $\theta$, defined by
\begin{equation}
k(u,\theta)=-L(u):=-\left(
\begin{array}
[c]{c}%
L_{1}(u)\\
\vdots\\
L_{l}(u)
\end{array}
\right)  \label{L(u)=shiftedPol}%
\end{equation}
where the shifted Legendre polynomials $L_{r}$ are as in Definition
\ref{DefLegendreshift}.

The SPLQ model (\ref{lmom_model}) may be written as
\begin{equation}
\mathcal{L}:=\bigcup_{\theta\in\Theta}L_{\theta}=\bigcup_{\theta\in\Theta
}\left\{  \mathbf{F}\text{ s.t. }\int_{0}^{1}L(u)F^{-1}(u)du=-f(\theta
)\right\}  . \label{lmom_model2}%
\end{equation}

Due to Proposition \ref{prop02} we may write equation (\ref{eq_lmm}) for
$r\geq2$ as follows, making use of the integrated shifted Legendre polynomials
$K_{r}$ \ in lieu of $L_{r}.$
\begin{equation}
-\mathbb{E}\left[  \frac{1}{r}\sum_{k=0}^{r-1}(-1)^{k}\dbinom{r-1}{k}%
X_{k:n}\right]  =\int_{0}^{1}K_{r}(u)\mathbf{F}^{-1}(du)=f_{r}(\theta).
\label{eq_lmm2}%
\end{equation}


\begin{example}
Turning back to Example \ref{example1}, we define $k$ and $f$ by
\[
k(u,\theta)=-\left(
\begin{array}
[c]{c}%
L_{2}(u)\\
L_{3}(u)\\
L_{4}(u)
\end{array}
\right)
\]
and
\[
f(\theta)=\left(
\begin{array}
[c]{c}%
f_{2}(\theta)\\
f_{3}(\theta)\\
f_{4}(\theta)
\end{array}
\right)  =\left(
\begin{array}
[c]{c}%
\sigma(1-2^{-1/\nu})\Gamma(1+1/\nu)\\
f_{2}(\theta)[3-2\frac{1-3^{-1/\nu}}{1-2^{-1/\nu}})]\\
f_{2}(\theta)[6+\frac{5(1-4^{-1/\nu})-10(1-3^{-1/\nu})}{1-2^{-1/\nu}})]
\end{array}
\right)
\]
where $\theta=(\sigma,\nu)\in\mathbb{R}_{+}^{\ast}\times\mathbb{R}_{+}^{\ast}$
and $u\in\lbrack0;1]$; hence (\ref{lmom_model}) holds.
\end{example}

\begin{example}
Similarly, in case we consider Example \ref{example2}, we define $k$ and $f$
by
\[
k(u,\theta)=-\left(
\begin{array}
[c]{c}%
L_{2}(u)\\
L_{3}(u)\\
L_{4}(u)
\end{array}
\right)
\]
and
\[
f(\theta)=\left(
\begin{array}
[c]{c}%
f_{2}(\theta)\\
f_{3}(\theta)\\
f_{4}(\theta)
\end{array}
\right)  =\left(
\begin{array}
[c]{c}%
\frac{\sigma}{(1+\nu)(2+\nu)}\\
f_{2}(\theta)\frac{1-\nu}{3+\nu}\\
f_{2}(\theta)\frac{(1-\nu)(2-\nu)}{(3+\nu)(4+\nu)}%
\end{array}
\right)
\]
where $\theta=(\sigma,\nu)\in\mathbb{R}_{+}^{\ast}\times\mathbb{R}$ and
$u\in\lbrack0;1],$ which also validates (\ref{lmom_model}).
\end{example}

\subsection{Extension to models defined by order statistics conditions}

\label{models_order}

The order statistics given by equation (\ref{eq:lmom}) can be written as
\[
\mathbb{E}[X_{j:r}]=\int_{0}^{1}P_{j:r}(u)F^{-1}(u)du
\]
where the polynomials $P_{j:r}$ are given by
\[
P_{j:r}(u)=\frac{r!}{(j-1)!(r-j)!}u^{j-1}(1-u)^{r-j}.
\]

Any linear combination of moments of order statistics can be written as
\[
-\sum_{i=1}^{r}a_{j}\mathbb{E}[X_{j:r}]=\int_{0}^{1}P_{a}(u)F^{-1}(u)du
\]
with coefficients $a_{j}$'s $\ $belonging to $\mathbb{R}$ and
\[
P_{a}(u)=-\sum_{i=1}^{r}a_{j}P_{j:r}(u).
\]

These models are SPLQ (see \ref{lmom_model}) with
\begin{equation}
\mathcal{L}:=\bigcup_{\theta}L_{\theta}=\bigcup_{\theta}\left\{  F\text{ s.t.
}\int_{0}^{1}P(u)F^{-1}(u)du=-f(\theta)\right\}  \label{order_model2}%
\end{equation}
where $P:u\in\lbrack0;1]\mapsto P(u)\in\mathbb{R}^{l}$ is an array of $l$ polynomials.

\begin{example}
Turning back to Example \ref{example3}, we define $k$ and $f$ by
\[
k(u,\theta) = \left(
\begin{array}
[c]{c}%
P_{1:3}(u)\\
P_{2:3}(u)\\
P_{3:3}(u)
\end{array}
\right)
\]
and
\[
f(\theta) = \left(
\begin{array}
[c]{c}%
\theta-\nu\\
\theta\\
\theta+\nu
\end{array}
\right)
\]
where $\theta\in\mathbb{R}, \nu>0$ and $u\in[0;1]$.
\end{example}


\section{Minimum of $\varphi$-divergence estimators}

Estimation, confidence regions and tests based on moment conditions models
have evolved over thirty years. Hansen and Owen respectively proposed the
generalized method of moments (GMM)\cite{hansen82} and the empirical
likelihood (EL) estimators \cite{owen90}. Newey and Smith \cite{neweysmith04}
introduced the generalized empirical likelihood (GEL) family of estimators
encompassing the previous estimators. They also proposed the dual versions of
the GEL estimators, the minimum discrepancy estimators (MD). These estimators
are the solution of the minimization of a divergence with constraints
corresponding to the model; see also Broniatowski and Keziou
\cite{broniatowski12} for an approach through duality and properties of the
inference under misspecification. In the quantiles framework, Gourieroux
proposed an adaptation of GMM estimators in \cite{gourieroux08} for a
parametric model seen through its quantile function $F^{-1}(t,\theta)$. In the
following, we will consider inference based on divergences in order to present
estimators for models defined by L-moments conditions.

\subsection{$\varphi$-divergences}

Let $\varphi:\mathbb{R}\rightarrow\lbrack0,+\infty]$ be a strictly convex
function with $\varphi(1)=0$ such that $\text{dom}(\varphi)=\{x\in
\mathbb{R}|\varphi(x)<\infty\}:=(a_{\varphi},b_{\varphi})$ with $a_{\varphi
}<1<b_{\varphi}$. If $F$ and $G$ are two $\sigma$-finite measures of
$(\mathbb{R},B(\mathbb{R}))$ such that $G$ is absolutely continuous with
respect to $F$, we define the divergence between $F$ and $G$ by :
\begin{equation}
D_{\varphi}(G,F)=\int_{\mathbb{R}}\varphi\left(  \frac{dG}{dF}(x)\right)
dF(x) \label{def_div}%
\end{equation}
where $\frac{dG}{dF}$ is the Radon-Nikodym derivative. It is clear that when
$F=G$, $D_{\varphi}(F,G)=0$. Furthermore, as $\varphi$ is supposed to be
strictly convex,
\[
D_{\varphi}(G,F)=0\text{ if and only if }F=G.
\]
These divergences were independently introduced by Csiszar \cite{csiszar63} or
Ali and Silvey \cite{ali66} in the context of probability measures. Definition
\ref{def_div} holds for any $\sigma$-finite measures even if our notation
refers to probability measures. Indeed in the sequel we will consider
divergences between quantile measure which are $\sigma$-finite but may be not
finite. See Liese \cite{liese87} who also considered divergences between
$\sigma$-finite measures.

\begin{example}
The class of power divergences parametrized by $\gamma\geq0$ is defined
through the functions
\[
x\mapsto\varphi_{\gamma}(x)=\frac{x^{\gamma}-\gamma x+\gamma-1}{\gamma
(\gamma-1)}.
\]
The domain of $\varphi_{\gamma}$ depends on $\gamma$. The Kullback-Leibler
divergence is associated to $x>0\mapsto\varphi_{1}(x)=x\log(x)-x+1$, the
modified Kullback-Leibler ($KL_{m}$) divergence to $x>0\mapsto\varphi
_{0}(x)=-\log(x)+x-1$, the $\chi^{2}$-divergence to $x\in\mathbb{R}%
\mapsto\varphi_{2}(x)=1/2(x-1)^{2}$, etc.
\end{example}

\subsection{M-estimates with L-moments constraints}

\subsubsection{Minimum of $\varphi$-divergences for probability measures}

\label{sectionM}

A plain approach to inference on $\theta$ consists in mimicking the empirical
minimum divergence one, substituting the linear constraints with respect to
the distribution by the corresponding linear constraints with respect to the
quantile measure, and minimizing the divergence between all probability
measures satisfying the constraint and the empirical measure $\mathbf{F}_{n}$
pertaining to the data set.\ More formally this yields to the following program.

Denote by $M$ the set of all probability measures defined on $\mathbb{R}$. For
a given p.m. $\mathbf{F}$ in $M$ we consider the submodel which consists in
all p.m's $\mathbf{G}$ in $M$, absolutely continuous with respect \ to $F$,
and which satisfy the constraints on their first L-moments for a given
$\theta\in\Theta$. Identifying a measure $\mathbf{G}$ with its distribution
function $G$ we define
\[
L_{\theta}^{(0)}(\mathbf{F})=\left\{ \mathbf{G}\in M\text{ s.t. }\mathbf{G}%
\ll\mathbf{F},\int_{0}^{1}L(t)G^{-1}(t)dt=-f(\theta)\right\} .
\]
Probability measures $\mathbf{G}$ satisfying the constraints and bearing their
mass on the sample points belong to $L_{\theta}^{(0)}(\mathbf{F}_{n}).$ For
any parameter $\theta\in\Theta$, the distance between $\mathbf{F}$ and the
submodel $L_{\theta}^{(0)}(\mathbf{F})$ is defined by
\[
D_{\varphi}(L_{\theta}^{(0)}(\mathbf{F}),\mathbf{F})=\inf_{\mathbf{G}\in
L_{\theta}^{(0)}(\mathbf{F})}D_{\varphi}(\mathbf{G},\mathbf{F}),
\]
and its plug-in estimator is
\[
D_{\varphi}(L_{\theta}^{(0)}(\mathbf{F}_{n}),\mathbf{F}_{n})=\inf
_{\mathbf{G}\in L_{\theta}^{(0)}(\mathbf{F}_{n})}D_{\varphi}(\mathbf{G}%
,\mathbf{F}_{n}).
\]
which measures the distance between the empirical measure $\mathbf{F}_{n}$ and
the class of all the probability measures supported by the sample and which
satisfy the L-moment conditions for a given $\theta.$

A natural estimator for $\theta$ may be defined by
\begin{equation}
\hat{\theta}_{n}^{(0)}=\arg\inf_{\theta\in\Theta}D_{\varphi}(L_{\theta}%
^{(0)}(\mathbf{F}_{n}),\mathbf{F}_{n})=\arg\inf_{\theta\in\Theta}%
\inf_{\mathbf{G}\in L_{\theta}^{(0)}(F_{n})}\frac{1}{n}\sum_{i=1}^{n}%
\varphi(n\mathbf{G}(x_{i})). \label{estimator0}%
\end{equation}
Unfortunately, existence of this estimator may not hold. Indeed, we cannot
assess that $L_{\theta}^{(0)}(\mathbf{F}_{n})$ is not empty : its elements are
multinomial distributions $\sum_{i=1}^{n}w_{i}\delta_{x_{i}}$ whose weights
are solutions of a family of $l-1$ polynomial algebraic equation of degree $l$
(with $n$ unknowns $w_{1},...,w_{n}$)%

\[
\sum_{i=1}^{n}K_{r}\left(  \sum_{a=1}^{i}w_{a}\right)  (x_{i+1:n}%
-x_{i:n})=-f_{r}(\theta);\text{ }1<r\leq l.
\]
To our knowledge, general conditions of existence for the solutions of such
problems do not exist even if we consider signed weights $w_{i}$.\newline
Bertail in \cite{bertail06} proposes a linearization of the constraint in
(\ref{estimator0}). We here prefer to switch to a different approach. If we
consider the L-moment equation (\ref{eq_lmm2}), we see that the quantile
function plays a similar role as the distribution function in the classical
moment equations. We will then change the functional to be minimized in order
to be able to use duality for the optimization step.\newline

\subsubsection{Minimum of $\varphi$-divergences for quantile measures}

\label{section_estlmom}

We have seen that the characterization of the L-moments given by the equation
(\ref{eq_lmm2}) uses the quantile measure $\mathbf{F}^{-1},$ which is defined
by the generalized inverse function of $F$. If $\mathbf{F}^{-1}$ is absolutely
continuous, we can define the quantile-density $q(u)=(F^{-1})^{\prime}(u)$.
This density was called "sparsity" function by Tukey \cite{tukey65} as it
represents the sparsity of the distribution at the cumulating weight
$u\in\lbrack0;1]$. This is clear when we look at the empirical version of this
measure which is composed by nothing but the increments of the sample. Some
other approach, handling properties of the inverse function of $(F^{-1}%
)^{\prime}$, have been proposed by Parzen \cite{parzen79}. He claims that the
inference procedures based on $(F^{-1})^{\prime}$ possesses inherent
robustness properties.\newline

Define
\[
K(u)=\left(
\begin{array}
[c]{c}%
K_{2}(u)\\
\vdots\\
K_{l}(u)
\end{array}
\right)
\]
and
\[
f(u):=f^{(2:l)}(u)=\left(
\begin{array}
[c]{c}%
f_{2}(u)\\
\vdots\\
f_{l}(u)
\end{array}
\right)  .
\]

For any $\theta$ in $\Theta$ the submodel which consists of all p.m's
$\mathbf{G}$\textbf{\ }with mass on the sample points is substituted by the
set of all quantile measures denoted $\mathbf{G}^{-1}$ which have masses on
subsets of $\left\{  1/n,2/n,..,1\right\}  $ and whose distribution functions
coincide with the generalized inverse functions of elements in $L_{\theta
}^{(0)}(\mathbf{F}_{n})$.\newline As in the case of divergence minimization
for models constrained by moment conditions, we will relax the positivity for
the masses of the quantile measures (see \cite{broniatowski12}). Let then $N$
be the class of all $\sigma-$finite signed measures on $\mathbb{R}$. Let
$L(u):=\left(  L_{2}(u),..,L_{l}(u)\right)  ^{T}$ for all $u$ in $\left(
0,1\right)  .$ Introducing signed measures makes sense when the domain of the
function $\varphi$ is not restricted to $\mathbb{R}^{+},$ as occurs for the
chi-square divergence $\varphi_{2}$ . Making use of equation (\ref{eq_lmm2})
define
\begin{align*}
L_{\theta}(\mathbf{F_{n}}^{-1}) &  :=\left\{  \mathbf{G}^{-1}\in N\text{ s.t.
}\mathbf{G}^{-1}\ll\mathbf{F}_{n}^{-1}\text{ and }\int_{0}^{1}L(u)G^{-1}%
(u)du=-f(\theta)\right\}  \\
&  =\left\{  \mathbf{G}^{-1}\in N\text{ s.t. }\mathbf{G}^{-1}\ll\mathbf{F}%
_{n}^{-1}\text{ and }\int_{0}^{1}K(u)\mathbf{G}^{-1}(du)=f(\theta)\right\}
\end{align*}
the family of all measures $\mathbf{G}^{-1}$ with support included in
$\left\{  1/n,2/n,..,1\right\}  $ which satisfy the $l-1$ constraints
pertaining to the L-moments; see (\ref{lmom_model2}). Note that when
$\mathbf{F}$ bears an atom then for large enough $n$ then $\mathbf{G}%
^{-1\text{ }}$in $L_{\theta}(\mathbf{F_{n}}^{-1})$ has a support strictly
included in $\left\{  1/n,2/n,..,1\right\}  $ .\newline Since the measure
$\mathbf{G}^{-1\text{ }}$ is not necessarily positive, its distribution
function $G^{-1}$ is not necessarily a generalized inverse of a function $G;$
we will however inherit  of the notation $G^{-1}$ from the case when
$\mathbf{G}^{-1\text{ }}$ is a positive measure to denote its distribution
function. If $\mathbf{G}^{-1}$ is positive, the mass of $\mathbf{G}^{-1\text{ }}$\ at point $i/n$ is a
spacing $y_{i+1:n}-y_{i:n}$ where $y_{i:n}$ is the $i-th$ order statistics
of the sample $y_{1},..,y_{n}$ generating the empirical distribution function
$G.$ 

A natural proposal for an estimation procedure in the SPLQ model is then to
consider the minimum of a $\varphi$-divergence between quantile measures
through
\begin{align}
\hat{\theta}_{n}  &  =\arg\inf_{\theta\in\Theta}\inf_{\mathbf{G}^{-1}\in
L_{\theta}(\mathbf{F_{n}}^{-1})}\int_{0}^{1}\varphi\left(  \frac
{d\mathbf{G}^{-1}}{d\mathbf{F}_{n}^{-1}}(u)\right)  \mathbf{F}_{n}%
^{-1}(du)\label{lmom_est0}\\
&  =\arg\inf_{\theta\in\Theta}\inf_{\substack{(y_{1},...,y_{n})\in
\mathbb{R}^{n}\text{ s.t.}\\\sum_{i=1}^{n-1}K(i/n)(y_{i+1}-y_{i})=f(\theta
)}}\sum_{i=1}^{n-1}\varphi\left(  \frac{y_{i+1}-y_{i}}{x_{i+1:n}-x_{i:n}%
}\right)  (x_{i+1:n}-x_{i:n}). \label{lmom_est}%
\end{align}

\begin{remark}
The estimation defined by (\ref{lmom_est0}) produces estimators $\hat{\theta
}_{n}$ which do not depend on the location of the sample, since a change the
sample $(x_{i}\mapsto x_{i}+a)_{i=1...n}$ produces, independently on the value
of $a$, the same measure $\mathbf{F}_{n}^{-1}$ whose mass on point $i/n$ is
the gap $x_{i+1:n}-x_{i:n}$. The minimum discrepancy estimators defined by
(\ref{lmom_est}) are invariant with respect to the location of the underlying
distribution of the data. Due to this fact, we consider the model defined by
L-moments conditions only through equations of the form (\ref{eq_lmm2}).
\end{remark}

Both the constraint and the divergence criterion are expressed in function of
$\mathbf{G}^{-1}$ and the constraint is linear with respect to this measure.
This allows to use classical duality results in order to efficiently compute
the estimator $\hat{\theta}_{n}$. Before that, we reformulate this criterion
as a minimization of an "energy" of transformation of the sample.


\section{Dual representations of the divergence under L-moment constraints}

\label{section5} The minimization of $\varphi$-divergences under linear
equality constraint is performed using Fenchel-Legendre duality. It transforms
the constrained problems into an unconstrained one in the space of Lagrangian
parameters. Let $\psi$ denote the Fenchel-Legendre transform of $\varphi$,
namely, for any $t\in\mathbb{R}$
\[
\psi(t):=\sup_{x\in\mathbb{R}}\left\{  tx-\varphi(x)\right\}  .
\]
Let us recall that $dom(\varphi)=(a_{\varphi},b_{\varphi})$. We can now
present a general duality result for the two optimization problems that
transform a constrained problem (possibly in an infinite dimensional space)
into an unconstrained one in $\mathbb{R}^{l}$.

Let $C:\Omega\rightarrow\mathbb{R}^{l}$ and $a\in\mathbb{R}^{l}$. Denote
\[
L_{C,a}=\left\{  g:\Omega\rightarrow\mathbb{R}\text{ s.t. }\int_{\Omega
}g(t)C(t)\mu(dt)=a\right\}  .
\]

\begin{proposition}
\label{prop2}Let $\mu$ be a $\sigma$-finite measure on $\Omega\subset
\mathbb{R}$. Let $C:\Omega\rightarrow\mathbb{R}^{l}$ be an array of functions
such that
\[
\int_{\Omega}\Vert C(t)\Vert\mu(dt)<\infty.
\]
If there exists some $g$ in $L_{C,a}$ such that $a_{\varphi}<g<b_{\varphi}$
$\mu$-a.s. then the duality gap is zero i.e.
\begin{equation}
\inf_{g\in L_{C,a}}\int_{\Omega}\varphi\left(  g\right)  d\mu=\sup_{\xi
\in\mathbb{R}^{l}}\langle\xi,a\rangle-\int_{\Omega}\psi(\langle\xi
,C(x)\rangle)\mu(dx). \label{sup dual}%
\end{equation}
Moreover, if $\psi$ is differentiable, if $\mu$ is positive and if there
exists a solution $\xi^{\ast}$ of the dual problem which is an interior point
of
\[
\left\{  \xi\in\mathbb{R}^{l}\text{ s.t. }\int_{\Omega}\psi(\langle
\xi,C(x)\rangle)\mu(dx)<\infty\right\}  ,
\]
then $\xi^{\ast}$ is the unique maximum in (\ref{sup dual}) and
\[
\int\psi^{\prime}(\left\langle \xi^{\ast},C(x)\right\rangle )C(x)\mu(dx)=a.
\]
Furthermore the mapping $a\mapsto\xi^{\ast}(a)$ is continuous.
\end{proposition}

\begin{proof}
The proof is delayed to the Appendix.
\end{proof}

\begin{remark}
When $\mathbf{G}^{-1}\ll\mathbf{F}^{-1}$ , denoting $g^{\ast}=d\mathbf{G}%
^{-1}/d\mathbf{F}^{-1}$ and assuming $g^{\ast}\in L_{K,f(\theta)}$ , and when
$\mu=\mathbf{F}^{-1}$ it holds%
\[
\int\varphi(g^{\ast})d\mu=D_{\varphi}\left(  \mathbf{G}^{-1},\mathbf{F}%
^{-1}\right)  .
\]

\end{remark}

\begin{remark}
Here, the classical assumption of finiteness of $\mu$ is replaced by
\[
\int_{\Omega}\Vert C(x)\Vert\mu(dx)<\infty
\]
which is needed for the application of the dominated convergence Theorem; also
we refer to the illuminating paper by Csisz\'{a}r and Mat\'{u}\v{s}
\cite{matus12} for the description of the geometric tools used in the proof of
Proposition \ref{prop2}.
\end{remark}

We now apply the above Proposition \ref{prop2} to the case when the array of
functions $C$ is equal to $K$, the measure $\mu$ is the quantile measure
$\mathbf{F}^{-1}$ pertaining to the distribution function $F$ of a probability
measure and when the class of functions $L_{C,a}$ is substituted by the class
of functions $\mathbf{dG}^{-1}\mathbf{/dF}^{-1}$ when defined. Let $\theta
\in\Theta$ and $F$ be fixed. Let us recall that for any reference cdf $F$
\begin{equation}
L_{\theta}(\mathbf{F}^{-1}):=\left\{  \mathbf{G}^{-1}\ll\mathbf{F}^{-1}\text{
s.t. }\int_{\mathbb{R}}K(u)\mathbf{G}^{-1}(du)=f(\theta)\right\}  .
\label{L_theta(^-)}%
\end{equation}

\begin{corollary}
\label{cor1}If there exists some $\mathbf{G}^{-1}$ in $L_{\theta}%
(\mathbf{F}^{-1})$ such that $a_{\varphi}<d\mathbf{G}^{-1}/d\mathbf{F}%
^{-1}<b_{\varphi}$ $\mathbf{F}^{-1}$-a.s. then
\begin{equation}
\inf_{\mathbf{G}^{-1}\in L_{\theta}(\mathbf{F}^{-1})}\int_{0}^{1}%
\varphi\left(  \frac{d\mathbf{G}^{-1}}{d\mathbf{F}^{-1}}\right)
d\mathbf{F}^{-1}=\sup_{\xi\in\mathbb{R}^{l}}\langle\xi,f(\theta)\rangle
-\int_{0}^{1}\psi(\langle\xi,K(u)\rangle)\mathbf{F}^{-1}(du).
\label{resultatCorrolaire1}%
\end{equation}
Moreover, if $\psi$ is differentiable and if there exists a solution
$\xi^{\ast}$ of the dual problem which is an interior point of%
\[
\left\{  \xi\in\mathbb{R}^{l}\text{ s.t. }\int_{\mathbb{R}}\psi(\langle
\xi,K(u)\rangle)\mathbf{F}^{-1}(du)<\infty\right\}  ,
\]
then $\xi^{\ast}$ is the unique maximum in (\ref{resultatCorrolaire1}) and
\[
\int\psi^{\prime\ast}(\langle\xi,K(u)\rangle)K(u)\mathbf{F}^{-1}%
(du)=f(\theta).
\]

\end{corollary}

\begin{remark}
\bigskip The above Corollary \ref{cor1} is the cornerstone for the plug-in
estimator of $D_{\varphi}\left(  \mathbf{G,F}\right)  .$
\end{remark}

Let us present an other application of the above Proposition \ref{prop2}
leading to the same dual problem. Denote by $\lambda$ the Lebesgue measure on
$\mathbb{R}$ and $L_{\theta}^{\prime}(F)$ be the set of all functions $g$
defined by
\[
L_{\theta}^{\prime}(F)=\left\{  g:\mathbb{R}\rightarrow\mathbb{R}\text{ s.t.
}\int_{\mathbb{R}}K(F(x))g(x)\lambda(dx)=f(\theta)\right\}  ,
\]
whenever non void.

\begin{corollary}
\label{cor2}If there exists some $g$ in $L_{\theta}^{\prime}(F)$ such that
$a_{\varphi}<g<b_{\varphi}$ $\lambda$-a.s. then
\begin{equation}
\inf_{g\in L_{\theta}^{\prime}(F)}\int_{\mathbb{R}}\varphi\left(  g\right)
d\lambda=\sup_{\xi\in\mathbb{R}^{l}}\langle\xi,f(\theta)\rangle-\int
_{\mathbb{R}}\psi(\langle\xi,K(F(x))\rangle)dx. \label{eq:511}%
\end{equation}
Moreover, if $\psi$ is differentiable and if there exists a solution
$\xi^{\ast}$ of the dual problem which is an interior point of
\[
\left\{  \xi\in\mathbb{R}^{l}\text{ s.t. }\int_{\mathbb{R}}\psi(\langle
\xi,K(F(x))\rangle)dx<\infty\right\}  ,
\]
then $\xi^{\ast}$ is the unique maximizer in (\ref{eq:511}). It satisfies
\begin{equation}
\int\psi^{\prime}(\langle\xi^{\ast},K(F(x))\rangle)dx=f(\theta)
\label{eq:grad1}%
\end{equation}

\end{corollary}

\begin{proof}
We will detail the proof of Corollary \ref{cor2}. Corollary \ref{cor1} is
proved similarly.\newline We apply the above Proposition \ref{prop2} for
$\Omega=\mathbb{R}$, $\mu=\lambda$, the array of functions $C$ substituted by
the array of functions $x\mapsto K(F(x))$ and $a=f(\theta)$.\newline
Consequently, the class of functions $g$ depends upon $F$, and $L_{C,a}%
=L_{\theta}^{\prime}(F).$We need then to show that
\[
\int_{\mathbb{R}}\Vert K(F(x))\Vert dx<\infty.
\]
Denote $K:=(K_{i_{1}},...,K_{i_{l}})$ with $i_{j}\geq2$ for all $j$. Recall
that from equation (\ref{jacobi})
\[
K_{i_{j}}(t)=-t(1-t)\frac{J_{i_{j}-2}^{(1,1)}(2t-1)}{i_{j}-1}.
\]
It is clear that there exists $C>0$ such that $\left\vert \frac{J_{i_{j}%
-2}^{(1,1)}(2t-1)}{i_{j}-1}\right\vert <C$. Hence
\[
\int_{\mathbb{R}}\Vert K(F(x))\Vert dx<lC\int_{\mathbb{R}}%
F(x)(1-F(x))dx<+\infty
\]
since $F$ is the cdf of a random variable with finite expectation. By applying
Proposition \ref{prop2}, it then holds
\[
\inf_{g\in L_{\theta}^{\prime\prime}(F)}\int_{\mathbb{R}}\varphi\left(
g\right)  d\lambda=\sup_{\xi\in\mathbb{R}^{l}}\langle\xi,f(\theta)\rangle
-\int_{\mathbb{R}}\psi(\langle\xi,K(F(x))\rangle)dx.
\]

\end{proof}

\begin{remark}
\label{remark30} If we consider the class of functions
\[
L_{\theta}^{\prime\prime}(F)=\left\{  T:\mathbb{R}\rightarrow\mathbb{R}\text{
s.t. $T$ derivable $\lambda-$a.e. and }\int_{\mathbb{R}}K(F(x))\frac
{dT}{d\lambda}(x)\lambda(dx)=f(\theta)\right\}  ,
\]
containing the functions $T:=x\mapsto\int_{-\infty}^{x}g(t)dt$ rather than the
class of functions $g$, it holds that $T\in L_{\theta}^{\prime\prime}(F)$ if
and only if $dT/d\lambda\in L_{\theta}^{\prime}(F)$. Therefore,
\[
\inf_{T\in L_{\theta}^{\prime\prime}(F)}\int_{\mathbb{R}}\varphi\left(
\frac{dT}{d\lambda}\right)  d\lambda=\inf_{g\in L_{\theta}^{\prime}(F)}%
\int_{\mathbb{R}}\varphi\left(  g\right)  d\lambda,
\]

This seemingly formal definition of the function $T$ makes sense since we can
view $T$ as a deformation function, as detailed in the following Section
\ref{deformation}.
\end{remark}

\section{Reformulation of divergence projections and extensions}

\label{deformation}

\subsection{Minimum of an energy of deformation}

\subsubsection{The case of models defined by moments constraints}

Let us suppose for a while that $\mathbf{F}$ and $\mathbf{G}$ are both
absolutely continuous with respect to the Lebesgue measure defined on
$\mathbb{R}$. Define the function $T=G\circ F^{-1}$. Then $T$ is derivable
a.e. and $T^{\prime}=\frac{dT}{d\lambda}$. It holds
\[
D_{\varphi}(\mathbf{G},\mathbf{F})=\int_{\mathbb{R}}\varphi\left(
\frac{d\mathbf{G}}{d\mathbf{F}}(x)\right)  \mathbf{F}(dx)=\int_{0}^{1}%
\varphi\left(  T^{\prime}(u)\right)  du
\]
even if $\mathbf{G}$ is not a positive measure, as far as the integrand in the
central term of the above display is defined.\newline The function $T$ can be
viewed as a measure of the deformation of $\mathbf{F}$ into $\mathbf{G}$ and
\[
E_{1}(T)=\int\varphi\left(  \frac{dT}{d\lambda}\right)  d\lambda
\]
as an energy of this deformation.

It can be seen that the absolute continuity assumption of both $\mathbf{F}$
and $\mathbf{G}$ with respect to the Lebesgue measure can be relaxed.

\begin{proposition}
\label{prop_moments}Let $F$ and $G$ be two arbitrary cdf's and $\lambda$ be
the Lebesgue measure. Let us define
\[
M_{\theta}(\mathbf{F})=\left\{  \mathbf{G}\ll\mathbf{F}\text{ s.t. }%
\int_{\mathbb{R}}g(x,\theta)\mathbf{G}(dx)=0\right\}
\]
and let $M_{\theta}^{\prime}(\mathbf{F})$ denote the class of all functions
$T$ which are a.e derivable on $[0;1]$ defined through%

\begin{equation}
M_{\theta}^{\prime}(\mathbf{F})=\left\{  T:[0;1]\rightarrow\mathbb{R}\text{
s.t. }\int_{0}^{1} g(F^{-1}(u),\theta)\frac{dT}{d\lambda}(u)\lambda(du)=0\right\}  . \label{Tmoments}%
\end{equation}
Then if there exists $T\in M_{\theta}^{\prime}(\mathbf{F})$ such that $a_{\varphi}<\frac{dT}{d\lambda}<b_{\varphi}$ and $\mathbf{G} \in M_{\theta}(\mathbf{F})$ such that $a_{\varphi}<\frac{d\mathbf{G}}{d\mathbf{F}}<b_{\varphi}$
\[
\inf_{G\in M_{\theta}(\mathbf{F})}\int_{\mathbb{R}}\varphi\left(
\frac{d\mathbf{G}}{d\mathbf{F}}\left(  x\right)  \right)  \mathbf{F}%
(dx)=\inf_{T\in M_{\theta}^{\prime}(\mathbf{F})}E_{1}(T).
\]

\end{proposition}

\begin{proof}
This results from Proposition \ref{prop2} applied twice.\newline First, if
$C=g(.,\theta)$, $a=0$, $\mu=\mathbf{F}$ and $g=d\mathbf{G}/d\mathbf{F}$, it
holds
\[
\inf_{G\in M_{\theta}(\mathbf{F})}\int_{\mathbb{R}}\varphi\left(
\frac{d\mathbf{G}}{d\mathbf{F}}\left(  x\right)  \right)  \mathbf{F}%
(dx)=\sup_{\xi\in\mathbb{R}^{l}} -\int_{\mathbb{R}} \psi\left(  \langle\xi,
g(x,\theta)\rangle\right)  \mathbf{F}(dx).
\]

Secondly, if $C=g(F^{-1}(.),\theta)$, $a=0$, $\mu=\lambda$ and $g=dT/d\lambda
$, it holds
\[
\inf_{T\in M_{\theta}^{\prime}(\mathbf{F})}\int_{0}^{1} \varphi\left(
\frac{dT}{d\lambda}\right)  d\lambda=\sup_{\xi\in\mathbb{R}^{l}} -\int_{0}^{1}
\psi\left(  \langle\xi, g(F^{-1}(u),\theta)\rangle\right)  \lambda(du).
\]

Lemma \ref{stieljes} concludes the proof.
\end{proof}

The estimators of minimum divergence used in \cite{neweysmith04} and
\cite{broniatowski12} can be expressed in terms of $T$ , introducing the
empirical distribution of the sample in place of the true unknown distribution
$\mathbf{F}_{\theta_{0}}.$ For each $\theta$ in $\Theta$ it holds
\[
\inf_{\mathbf{G}\in M_{\theta}(\mathbf{F}_{n})}\int_{\mathbb{R}}\varphi\left(
\frac{d\mathbf{G}}{d\mathbf{F}_{n}}\right)  \mathbf{F}_{n}(dx)=\inf_{T\in
M_{\theta}^{\prime}(\mathbf{F}_{n})}E_{1}(T)
\]
and
\[
\theta_{n}:=\arg\inf_{\theta\in\Theta}\inf_{T\in M_{\theta}^{\prime
}(\mathbf{F}_{n})}E_{1}(T).
\]

\begin{remark}
Note that if $T\in M_{\theta}^{\prime}(\mathbf{F}_{n})$, $T:[0;1]\rightarrow
[0;1]$ is $\lambda$-a.e. derivable and verifies
\[
\sum_{i=1}^{n-1}g(x_{i:n},\theta)\left(  T\left(  \frac{i+1}{n}\right)
-T\left(  \frac{i}{n}\right)  \right)  =0.
\]

\end{remark}

The plug-in estimator that realizes the minimum of the divergence between a
given distribution and the submodel $\mathcal{M}_{\theta}$ results from the
minimum of an energy of a deformation of the uniform grid on $\left[
0,1\right]  $ under constraints envolving \ the observed sample.\ \ Therefore
the classical minimum divergence approach under moment conditions turns out to
be a tranformation of the uniform measure on the sample points, represented by
the uniform grid on $\left[  0,1\right]  $ onto \ a projected measure on the
same sample points, and the projected measure $\mathbf{G}_{n}$ which solves
the primal problem has support $x_{1},..,x_{n}$ and has a distribution
function $G_{n}$ $=T(F_{n})$ where $T$ solves
\[
\inf_{T\in M_{\theta}^{\prime}(\mathbf{F}_{n})}E_{1}(T).
\]
\ \ \ \ \ 

\bigskip Turning now to the case of models defined by L-moments, we will now
see that the approach of Section \ref{section_estlmom} consists in minimizing
a deformation of the points of the distribution of interest instead of the weights.

\subsubsection{The case of models defined by L-moment constraints}

Similarly as for the case of models defined by moment constraints we now see
that the solution of the minimum divergence problem (primal problem) holds
without assuming $\mathbf{F}^{-1}$ absolutely continuous with respect to the
Lebesgue measure.

\begin{proposition}
\label{prop10}Let $F$ and $G$ be two arbitrary cdf's. Let $L_{\theta}%
^{\prime\prime}(\mathbf{F}^{-1})$ denote the class of all functions $T$ which
are a.e derivable on $\mathbb{R}$ defined through%

\begin{equation*}
L_{\theta}^{\prime\prime}(\mathbf{F}^{-1})=\left\{  T:\mathbb{R}%
\rightarrow\mathbb{R}\text{ s.t. }\int_{\mathbb{R}}K(F(x))\frac{dT}{d\lambda
}(x)\lambda(dx)=f(\theta)\right\}  . 
\end{equation*}
Then, with $L_{\theta}(\mathbf{F}^{-1})$ defined in (\ref{L_theta(^-)}), if there exists $T\in L_{\theta}^{\prime\prime}(\mathbf{F}^{-1})$ such that $a_{\varphi}<\frac{dT}{d\lambda}<b_{\varphi}$ and $\mathbf{G}^{-1} \in L_{\theta}(\mathbf{F}^{-1})$ such that $a_{\varphi}<\frac{d\mathbf{G}^{-1}}{d\mathbf{F}^{-1}}<b_{\varphi}$
\[
\inf_{\mathbf{G}^{-1}\in L_{\theta}(\mathbf{F}^{-1})}\int_{0}^{1}%
\varphi\left(  \frac{d\mathbf{G}^{-1}}{d\mathbf{F}^{-1}}\left(  u\right)
\right)  \mathbf{F}^{-1}(du)=\inf_{T\in L_{\theta}^{\prime\prime}%
(\mathbf{F}^{-1})}\int_{\mathbb{R}}\varphi\left(  \frac{dT}{d\lambda}\right)
d\lambda.
\]

\end{proposition}

\begin{proof}
This results from a combination of Corollaries \ref{cor1} and \ref{cor2}.
\end{proof}

In the following, we consider the estimator of $\theta$
\begin{equation}
\hat{\theta}_{n}=\arg\inf_{\theta\in\Theta}\inf_{T\in L_{\theta}^{\prime
\prime}(\mathbf{F}_{n}^{-1})}\int_{\mathbb{R}}\varphi\left(  \frac
{dT}{d\lambda}\right)  d\lambda. \label{lmde2_1}%
\end{equation}

The estimator $\hat{\theta}_{n}$ defined in (\ref{lmde2_1}) coincides with
(\ref{lmom_est0}) thanks to the above Proposition \ref{prop10}.

\begin{remark}
$\cup_{\theta}L_{\theta}(\mathbf{F}^{-1})$ and $\cup_{\theta}L_{\theta
}^{\prime\prime}(\mathbf{F}^{-1})$ both represent the same model with
L-moments constraints, seen through a reference measure $\mathbf{F}^{-1}$.
This model is either expressed as the space of quantile measures absolutely
continuous with respect to $\mathbf{F}^{-1}$ satisfying the L-moment
constraints or as the space of all deformations $\mathbf{F}^{-1}\rightarrow$
$T\circ F^{-1}$ of the reference measure $\mathbf{F}^{-1}$ such that the
deformed measure satisfies the L-moment constraints. In the second point of
view $T$ is derivable $\lambda$-a.e. even if the reference measure is
$\mathbf{F}_{n}^{-1}$.
\end{remark}

\begin{remark}
For the set of deformations $L_{\theta}^{\prime\prime}(\mathbf{F}_{n}^{-1})$
(whenever non void), the duality for finite distributions is expressed through
the following equality :
\[
\inf_{T\in L_{\theta}^{\prime\prime}(\mathbf{F}_{n}^{-1})}\int\varphi\left(
\frac{dT}{d\lambda}\right)  d\lambda=\sup_{\xi\in\mathbb{R}^{l}}\xi
^{T}f(\theta)-\sum_{i=1}^{n-1}\psi\left(  \xi^{T}K\left(  \frac{i}{n}\right)
\right)  (x_{i+1:n}-x_{i:n}).
\]
Remark that we incorporate the requirement that for any $T$ in the model
$L_{\theta}^{\prime\prime}(\mathbf{F}_{n}^{-1})$ , $a_{\varphi}<\frac{dT_{1}%
}{d\lambda}<b_{\varphi}$ $\lambda$-a.s. holds.
\end{remark}

\begin{example}
If we consider the $\chi^{2}$-divergence $\varphi(x)=\frac{(x-1)^{2}}{2}$,
then $\psi(t)=\frac{1}{2}t^{2}+t$ and the solution $\xi_{1}^{\ast}$ of the
equation (\ref{eq:grad1}) is
\[
\xi_{1}^{\ast}=\Omega^{-1}\left(  f(\theta)-\int K(F(x))d\lambda\right)
\]
with
\[
\Omega=\int K(F(x))K(F(x))^{T}d\lambda.
\]
If we set $\Omega_{n}=\int K(F_{n}(x))K(F_{n}(x))d\lambda$, the estimator
shares similarities with the GMM estimator. Indeed
\[
\hat{\theta}_{n}=\arg\inf_{\theta\in\Theta}\left(  f(\theta)-\int
K(F_{n}(x))d\lambda\right)  \Omega_{n}^{-1}\left(  f(\theta)-\int
K(F_{n}(x))d\lambda\right)  .
\]
This divergence should thus be favored for its fast implementation.
\end{example}

\begin{remark}
We did not consider the constraints of positivity classically assumed in
moment estimating equations for the sake of simplicity of dual
representations. We could suppose that the transformation $T$ is an increasing
mapping. It would be the case if, for example, the divergence chosen is the
Kullback-Leibler one. Indeed, in this case, problem (\ref{lmde2_1}) is well
defined since $\varphi(x)=+\infty$ for all $x\leq0$.
\end{remark}

\subsection{Transportation functionals and multivariate generalization}

The notion of a deformation which was introduced in the above section is close
to the notion of a transportation. The reformulation presented in Proposition
\ref{prop10} calls for a natural extension in this respect. Let us recall the
definition of a transportation in $\mathbb{R}$.

\begin{definition}
The pushforward measure of $\mathbf{F}$ through $T$ is the measure denoted by
$T\#\mathbf{F}$ satisfying
\[
T\#\mathbf{F}(B)=\mathbf{F}(T^{-1}(B))\text{ for every Borel subset $B$ of
$\mathbb{R}$.}%
\]
$T$ is said to be a transportation map between $\mathbf{F}$ and $\mathbf{G}$
if $T\#\mathbf{F}=\mathbf{G}$. If $X$ and $Y$ are associated with respective
cdf $F$ and $G$ then $T(X)=_{d}Y.$
\end{definition}

We write $L_{\theta}$ (equation (\ref{lmom_model2})) as a space of $\sigma
$-measures
\[
L_{\theta}=\left\{  \mathbf{G}\text{ }\in M\text{ s.t. }\int_{0}^{1}%
L(u)G^{-1}(u)du=-f(\theta)\right\} .
\]
Let furthermore $\mathcal{C}_{A}$ denote the space of absolutely continuous
functions defined on $\mathbb{R}$. It follows that an alternative to the
estimator (\ref{lmde2_1}) may be defined by
\begin{equation}
\hat{\theta}_{n}^{(tr)}=\arg\inf_{\theta\in\Theta}\inf_{T\in\mathcal{C}%
_{A}:T\#\mathbf{F}_{n}\in L_{\theta}}\int_{\mathbb{R}}\varphi\left(  \frac
{dT}{d\lambda}\right)  d\lambda\label{eq:lmde2_2}%
\end{equation}
where

\begin{itemize}
\item $\mathbf{F}_{n}$ is the empirical measure on the observed sample
$x_{1},...,x_{n}$

\item $E(T):=\int_{\mathbb{R}}\varphi\left(  \frac{dT}{d\lambda}\right)
d\lambda$ stands for the energy which transports $\mathbf{F}_{n}$ onto some
$\mathbf{G}$.
\end{itemize}

We can give a rewriting of this transport estimator similar to Equation
(\ref{lmom_est}).

\begin{proposition}
\label{prop40}
If there exists some absolutely continuous $T_{0}$ such that $T_{0}%
\#\mathbf{F}_{n}\in L_{\theta}$ and $a_{\varphi}<\frac{dT_{0}}{d\lambda
}<b_{\varphi}$, then
\begin{equation}
\hat{\theta}_{n}^{(tr)}=\arg\inf_{\theta\in\Theta}\inf_{\substack{y\in
\mathbb{R}^{n}\\\sum_{i=1}^{n-1}K(i/n)(y_{i+1:n}-y_{i:n})=f(\theta)}%
}D_{\varphi}(x,y) \label{Prop1}%
\end{equation}
with
\[
D_{\varphi}(x,y) = \sum_{i=1}^{n-1}\varphi\left(  \frac{y_{i+1}-y_{i}}{x_{i+1:n}-x_{i:n}%
}\right)  (x_{i+1:n}-x_{i:n}).
\]
If moreover, $\varphi(x)=+\infty$ for any $x\leq0$ then
\begin{equation}
\hat{\theta}_{n}^{(tr)}=\hat{\theta}_{n}\label{Prop2}%
\end{equation}

\end{proposition}

\begin{proof}
The proof is postponed to the Appendix.
\end{proof}

\begin{remark}
The fact that $T$ is absolutely continuous is necessary.\ Indeed, stating
\[
\hat{\theta}_{n}^{(tr0)}:=\arg\inf_{\theta\in\Theta}\inf_{T:T\#\mathbf{F}%
_{n}\in L_{\theta}}\int_{\mathbb{R}}\varphi\left(  \frac{dT}{d\lambda}\right)
d\lambda
\]
may not lead to a well defined estimator; consider any discrete uniform
distribution in $L_{\theta}$ (i.e any distribution in the submodel $L_{\theta
}(\mathbf{F}_{n}^{-1})$). Let us denote its support by $y:=\{y_{1}%
,...,y_{n}\},$ and define $T_{y}$ a.e. derivable such that
\[
T_{y}(x)=\left\{
\begin{array}
[c]{l}%
y_{i} \text{ if $x=x_{i}$}\\
x\text{ otherwise}%
\end{array}
\right.  ,
\]
then $T_{y}\#\mathbf{F}_{n}\in L_{\theta}$ and
\[
\int_{\mathbb{R}}\varphi\left(  \frac{dT_{y}}{d\lambda}\right)  d\lambda=0
\]
since $\varphi(1)=0$. So when $L_{\theta}(\mathbf{F}_{n}^{-1})$ is not reduced
to a unique measure, this estimator is undefined : the solution of the infimum
problem is not unique.
\end{remark}


In transportation theory, it is customary to define a cost function instead of
an energy function. Given a convex cost function $c:\mathbb{R}\times
\mathbb{R}\rightarrow\mathbb{R}$, an alternative version to (\ref{lmde2_1})
is
\begin{equation}
\hat{\theta}_{n}=\arg\inf_{\theta\in\Theta}\inf_{T:T\#\mathbf{F}_{n}\in
L_{\theta}}\int_{\mathbb{R}}c(x,T(x))\mathbf{F}_{n}(dx). \label{eq:lmde2_3}%
\end{equation}

\begin{remark}
Whereas the estimator given by Equation (\ref{lmde2_1}) minimizes an energy
expressed in function of $T^{\prime}$ (the estimation process then penalizes
big values of $T^{\prime}$), the optimal transportation estimator depends on
the function $T$ itself and penalizes the distance between each $x_{i}$ and
$T(x_{i})$ i.e. the "initial" state and the deformed state.
\end{remark}

\begin{example}
The following estimator stems from the optimal transportation problem
(\ref{eq:lmde2_3}) in the context of models constrained by L-moments
equations.\newline Consider the cost function $c(x,y)=(x-y)^{2}$. The
transportation problem reduces to (see e.g. \cite{villani04})
\[
\inf_{T:T\#\mathbf{F_{n}}\in L_{\theta}}\int_{\mathbb{R}}\left( x-T(x)\right)
^{2}\mathbf{F}_{n}(dx):=\inf_{\mathbf{G}\in L_{\theta}}W_{2}(\mathbf{F}%
_{n},\mathbf{G})^{2}=\inf_{\mathbf{G}\in L_{\theta}}\int_{0}^{1}\left|
F_{n}^{-1}(t)-G^{-1}(t)\right| ^{2}dt,
\]
$W_{2}$ is called the Wasserstein distance. The estimator (\ref{eq:lmde2_3})
will then be defined by
\begin{align*}
\hat{\theta}_{n}  &  :=\arg\inf_{\theta\in\Theta}\inf_{T:T\#\mathbf{F_{n}}\in
L_{\theta}}\int_{\mathbb{R}}\left( x-T(x)\right) ^{2}\mathbf{F}_{n}(dx)\\
&  =\arg\inf_{\theta\in\Theta}\min_{\substack{y\in\mathbb{R}^{n}\\\sum
_{i=1}^{n-1}K(i/n)(y_{i+1:n}-y_{i:n})=f(\theta)}}\frac{1}{n}\sum_{i=1}%
^{n}|x_{i:n}-y_{i:n}|^{2}
\end{align*}
with $l_{r}$ given by equation (\ref{empirical_vstat}).\newline

\end{example}

As transportation is well defined for measure in $\mathbb{R}^{d}$ in contrast
with quantile measures, this may appear as a way to generalize L-moments
constrained models and associated estimators of the form (\ref{eq:lmde2_2});
we could also consider estimators of the form (\ref{eq:lmde2_3}), importing
henceforth optimal transportation concepts in the field of multivariate
quantile models; see \cite{decurninge14}.

\subsection{Relation to elasticity theory}

It may be of interest for the statistician to observe that, besides the
probabilistic context of semiparametrics, the minimization of a $\varphi$
divergence over a class of functions defined by L-moments (see (\ref{lmde2_1}%
)) is in the same vein as finding the deformation of a solid under a given
force $L$ and given boundary constraints. Let us consider a solid defining a
domain $\Omega\subset\mathbb{R}^{3}$. This solid can be deformed under the
action of volumetric or surface forces. This deformation can be described by a
function $T:\Omega\rightarrow\mathbb{R}^{3}$. The deformed solid will be
defined on the volume $T(\Omega)$. The gradient of deformation is then $\nabla
T$.\newline The general equations describing the equilibrium of the solid
under volumetric forces $L$ defined on $\Omega$ read (we omit boundary
forces)
\[
-\text{div}S=L
\]
where $S$ is a tensor describing the configuration of the solid \cite{blanc07}%
. Hyper-elasticity is often assumed i.e. the solid is supposed to dissipate no
energy during the deformation. In mathematical terms, this means the existence
of a function $\varphi$ such that
\[
S(T)=\frac{\partial\varphi}{\partial T}(T).
\]
From these above relations, the energy of deformation is expressed on the form
\cite{lebris05}\cite{blanc07}
\[
\mathcal{E}(T)=\int_{\Omega}\varphi(\nabla T(x))dx-\int_{\Omega}L(x).T(x)dx.
\]
$\varphi$ is usually convex and represents physical properties of the solid.
It is then customary in mechanical physics to assume the principle of least
action and to study the $T$ minimizing the variational problem
\[
\inf_{T\text{ admissible}}\mathcal{E}(T).
\]
The space of admissible $T$ describes the constraints, such as boundary
conditions. If we could write the volumetric force term (namely the right hand
side of $\mathcal{E}(T)$) as fixed constraints, we remark similarities with
the estimation given by equation \ref{lmde2_1}
\[
\hat{\theta}_{n}=\arg\inf_{\theta\in\Theta}\inf_{\int_{\Omega}%
L(x).T(x)dx=f(\theta)}\int_{\Omega}\varphi\left(  \nabla T(x)\right)  dx.
\]

Moreover, microscopic and macroscopic scales can be related through
convergence results. Let us present the microscopic models of the same solid
represented by $N$ particles $x_{1},...,x_{N}$, corresponding for example to
the intersection of $\Omega$ with a lattice of scale $\epsilon$. If $V$
denotes an interaction potential, the energy of the solid subjected to a
deformation $T$ would be
\[
\mathcal{E}_{N}(T)=\frac{\epsilon^{3}}{2}\sum_{i=1}^{N}\sum_{j\neq i}V\left(
\frac{T(x_{i})-T(x_{j})}{\epsilon}\right)  -\sum_{i=1}^{N}L(x_{i}).T(x_{i})
\]
where for any $3 \times3$ matrix $M$
\[
\varphi(M) = \frac{1}{2}\sum_{k\in\mathbb{Z}^{3}\backslash\{0\}} V(Mk).
\]
Under some assumptions (see \cite{blanc02}), it can be proved that if
$\epsilon\rightarrow0$ (i.e. $N\rightarrow\infty$), then
\[
\mathcal{E}_{N}(T)\rightarrow_{N\rightarrow\infty}\mathcal{E}(T).
\]
This short account may give us some intuition about the present estimation


\section{Asymptotic properties of the L-moment estimators}

\label{section7} In this section, we study the convergence of the estimator
given by the equation (\ref{lmde2_1}). The proof of the two asymptotic
theorems are postponed to the Appendix.

\begin{theorem}
\label{th1}Let $x_{1},...,x_{n}$ be an observed sample drawn iid from a
distribution $F_{0}$ with finite variance. Assume that

\begin{itemize}
\item there exists $\theta_{0}$ such that $F_{0} \in L_{\theta_{0}}$,
$\theta_{0}$ is the unique solution of the equation $f(\theta)=f(\theta_{0})$

\item $f$ is continuous and $\Theta\subset\mathbb{R}^{d}$ is compact

\item the matrix $\Omega_{0} = \int K(F_{0}(x)) K(F_{0}(x))^{T} dx$ is non singular.
\end{itemize}

Then
\[
\hat{\theta}_{n}\rightarrow\theta_{0}\text{ in probability as }n\rightarrow
\infty.
\]

\end{theorem}

We may now turn to the limit distribution of the estimator. Let

\begin{itemize}
\item $J_{0}=J_{f}(\theta_{0})$ be the Jacobian of $f$ with respect to
$\theta$ in $\theta_{0}$

\item $M=(J_{0}^{T}\Omega^{-1}J_{0})^{-1}$

\item $H=MJ_{0}^{T}\Omega^{-1}$

\item $P=\Omega^{-1}-\Omega^{-1}J_{0}MJ_{0}^{T}\Omega^{-1}$
\end{itemize}

\begin{theorem}
\label{theorem_asymp} Let $x_{1},...,x_{n}$ be an observed sample drawn iid
from a distribution $F_{0}$ with finite variance. We assume that the
hypotheses of Theorem \ref{th1} holds. Moreover, we assume that

\begin{itemize}
\item $\theta_{0}\in int(\Theta)$

\item $J_{0}$ has full rank

\item $f$ is continuously differentiable in a neighborhood of $\theta_{0}$
\end{itemize}

Then,%
\[
\sqrt{n}\left(
\begin{array}
[c]{ccc}%
\hat{\theta}_{n}-\theta_{0} &  & \\
\hat{\xi}_{n} &  &
\end{array}
\right)  \rightarrow_{d}\mathcal{N}_{d+l}\left(  0,\left(
\begin{array}
[c]{cc}%
H\Sigma H^{T} & 0\\
0 & P\Sigma P^{T}%
\end{array}
\right)  \right)
\]

\end{theorem}

The estimator of the minimum of the divergence from $\mathbf{F}$ onto the
model, namely $2n\left[  \hat{\xi}_{n}^{T}f(\hat{\theta}_{n})-\int\psi
(\hat{\xi}_{n}^{T}K(F_{n}(x))dx\right]  ,$ does not converge to a $\chi^{2}%
$-distribution as in the case of moment condition models \cite{neweysmith04}.
However, we can state an alternative result.

\begin{corollary}
Let us assume that the hypotheses of Theorem \ref{theorem_asymp} hold.\newline
Let $S_{n}:=n\hat{\xi}_{n}^{T}(P_{n}\Sigma_{n}P_{n}^{T})^{-1}\hat{\xi}_{n}$
with $P_{n}$ and $\Sigma_{n}$ the respective empirical versions of $P$ and
$\Sigma$.\newline If $P\Sigma P$ is non singular then%
\[
S_{n}\rightarrow_{d}\chi^{2}(l)
\]
where $\chi^{2}(l)$ denotes a chi-square distribution with $l$ degrees of freedom.
\end{corollary}

\begin{proof}
From Theorem \ref{theorem_asymp}, we have that%
\[
n^{1/2}\hat{\xi}_{n}\rightarrow_{d}X=\mathcal{N}_{l}(0,P\Sigma P)
\]
where $X$ denotes such a multivariate Gaussian random vector.\newline
Furthermore%
\[
P_{n}\Sigma_{n}P_{n}\rightarrow_{p}P\Sigma P.
\]
Hence, for $n$ large enough, $P_{n}\Sigma_{n}P_{n}$ is invertible and by
Slutsky Theorem
\[
n\hat{\xi}_{n}^{T}(P_{n}\Sigma_{n}P_{n})^{-1}\hat{\xi}_{n}\rightarrow_{p}%
X^{T}X=_{d}\chi^{2}(l).
\]

\end{proof}

Since the weak convergence of $S_{n}$ to a chi-square distribution is
independent of the value of $\theta_{0}$, this result may be used in order to
build confidence regions related to the semi-parametric model.




\section{Numerical applications : Inference for Generalized Pareto family}

\subsection{Presentation}

The Generalized Pareto Distributions (GPD) are known to be heavy-tailed
distributions. They are classically parametrized by a location parameter $m$,
which we assume to be $0$, a scale parameter $\sigma$ and a shape parameter
$\nu$. They can be defined through their density :
\[
f_{\sigma,\nu}(x)=\left\{
\begin{array}
[c]{ll}%
\frac{1}{\sigma}\left(  1+\nu\frac{x}{\sigma}\right)  ^{-1-1/\nu
}\mathds{1}_{x>0} & \text{ if $\nu>0$}\\
\frac{1}{\sigma}\exp\left(  \frac{x}{\sigma}\right)  \mathds{1}_{x>0} & \text{
if $\nu=0$}\\
\frac{1}{\sigma}\left(  1+\nu\frac{x}{\sigma}\right)  ^{-1-1/\nu
}\mathds{1}_{-\sigma/\nu>x>0} & \text{ if $\nu<0$}%
\end{array}
\right.
\]
Let us remark that if $\nu\geq1$, the GPD does not have a finite expectation.
We perform different estimations of the scale and the shape parameter of a GPD
from samples with size $n=100$. \newline We will estimate the parameters in
the model composed by the distributions of all r.v's $X$ whose second, third
and fourth L-moments verify
\begin{equation}
\label{eq_lmom_gpd}\left\{
\begin{array}
[c]{lll}%
\lambda_{2} & = & \frac{\sigma}{(1-\nu)(2-\nu)}\\
\frac{\lambda_{3}}{\lambda_{2}} & = & \frac{1+\nu}{3-\nu}\\
\frac{\lambda_{4}}{\lambda_{2}} & = & \frac{(1+\nu)(2+\nu)}{(3-\nu)(4-\nu)}%
\end{array}
\right.
\end{equation}
for any $\sigma>0,\nu\in\mathbb{R}$. These distributions share their first
L-moments with those of a GPD with scale and shape parameter $\sigma$ and
$\nu$ (see \cite{hosking90}). This estimation will be compared with classical
parametric estimators detailed hereafter.

\subsection{Moments and L-moments calculus}

The variance and the skewness of the GPD are given by
\[
\left\{
\begin{array}
[c]{lll}%
var & = & \mathbb{E}[(X-\mathbb{E}[X])^{2}]=\frac{\sigma^{2}}{(1-\nu
)^{2}(1-2\nu)}\\
t_{3} & = & \mathbb{E}[\left(  \frac{X-\mathbb{E}[X]}{\mathbb{E}%
[(X-\mathbb{E}[X])^{2}]}\right)  ^{3}]=\frac{2(1+\nu)\sqrt{1-2\nu}}{1-3\nu}%
\end{array}
\right.
\]

Let us remark that $var$ and $t_{3}$ respectively exist since $\nu<1/2$ and
$\nu<1/3.$\newline On the other hand, the first L-moments are given by
equation \ref{eq_lmom_gpd}. Assuming $\nu<1$ entails existence of the L-moments.

\subsection{Simulations}

We perform $N=500$ runs of the following estimators

\begin{itemize}
\item the estimation proposed in this article (equation (\ref{lmde2_1})) for
the $\chi^{2}$-divergence and the modified Kullback ($KL_{m}$) divergence with
the constraints estimated on the L-moments of order $2,3,4$

\item the estimate defined through the L-moment method, based on the empirical
second L-moment $\hat{\lambda}_{2}$ and the fourth L-moment ratio $\hat
{\tau_{4}}=\frac{\lambda_{4}}{\lambda_{2}}$
\begin{align*}
\hat{\nu}  &  =\frac{7\hat{\tau_{4}}+3-\sqrt{(\hat{\tau_{4}}^{2}+98\hat
{\tau_{4}}+1}}{2(\hat{\tau_{4}}-1)}\\
\hat{\sigma}  &  =\hat{\lambda}_{2}(1-\hat{\nu})(2-\hat{\nu})
\end{align*}

\item the estimate defined through the moment method estimated from the
empirical variance $\hat{var}$ and skewness $\hat{t}_{3}$
\begin{align*}
\hat{\nu}  &  =\frac{2(1+\hat{t}_{3})\sqrt{1-2\hat{t}_{3}}}{1-3\hat{t}_{3}}\\
\hat{\sigma}  &  =\sqrt{\hat{var}(1-\hat{t}_{3})^{2}(1-2\hat{t}_{3})}%
\end{align*}

\item the MLE defined in the GPD family
\end{itemize}

We present the following different features for any of the above estimators

\begin{itemize}
\item the mean of the $N$ estimates based on the $N$ runs

\item the median of the $N$ estimates based on the $N$ runs

\item the standard deviation of the \ $N$ estimates

\item the $L_{1}$ distance between the estimated generalized Pareto density
and the true density, namely
\[
\int_{x\geq0}|f_{\hat{\sigma},\hat{\nu}}(x)-f_{\sigma,\nu}(x)|dx
\]
which, by Scheff\'{e} Lemma, equals twice the maximum error committed
substituting $f_{\sigma,\nu}$ by $f_{\hat{\sigma},\hat{\nu}}$
\[
\int_{x\geq0}|f_{\hat{\sigma},\hat{\nu}}(x)-f_{\sigma,\nu}(x)|dx=2\sup
_{A\in\mathcal{B}(\mathbb{R})}\left\vert \int_{A}f_{\hat{\sigma},\hat{\nu}%
}(x)-\int_{A}f_{\sigma,\nu}(x)|dx\right\vert .
\]

\end{itemize}

Finally, we present four different scenarios which illustrate robusness
properties of any of the above estimators, as well as their behavior under misspecification:

\begin{itemize}
\item a first scenario without outliers : samples of size $30$ or $100$ are
drawn from a GPD

\item two more scenarios with $10\%$ outliers : samples of size $27$ or $90$
are drawn from a GPD. The remaining points are drawn from a Dirac the value of
which depends on the shape parameter

\item a fourth scenario without outliers but with misspecification : samples
of size $30$ or $100$ are drawn from a Weibull distribution.
\end{itemize}

\begin{table}[ptb]
\center
\begin{tabular}
[c]{|c|c||c|c|c||c|c|c|}\cline{3-8}%
\multicolumn{2}{c||}{} & \multicolumn{3}{|c|}{$n=30$} &
\multicolumn{3}{|c|}{$n=100$}\\\hline
Estimation method & Parameter & Mean & Median & StD & Mean & Median &
StD\\\hline
$\chi^{2}$-divergence & $\sigma$ & 4.68 & 4.41 & 2.52 & 3.80 & 3.75 &
0.90\\\hline
$KL_{m}$-divergence & $\sigma$ & 6.44 & 4.77 & 8.02 & 4.08 & 3.95 &
4.00\\\hline
L-moment method & $\sigma$ & 5.67 & 4.98 & 3.44 & 3.96 & 3.80 & 1.09\\\hline
Moment method & $\sigma$ & 17.17 & 10.45 & 62.95 & 17.15 & 11.64 &
19.52\\\hline
MLE & $\sigma$ & 3.33 & 3.17 & 1.14 & 3.08 & 3.07 & 0.57\\\hline\hline
$\chi^{2}$-divergence & $\nu$ & 0.38 & 0.39 & 0.24 & 0.55 & 0.55 &
0.16\\\hline
$KL_{m}$-divergence & $\nu$ & 0.37 & 0.38 & 0.24 & 0.38 & 0.37 & 0.16\\\hline
L-moment method & $\nu$ & 0.33 & 0.38 & 0.31 & 0.54 & 0.56 & 0.18\\\hline
Moment method & $\nu$ & 0.08 & 0.12 & 0.12 & 0.21 & 0.22 & 0.06\\\hline
MLE & $\nu$ & 0.61 & 0.63 & 0.33 & 0.68 & 0.69 & 0.17\\\hline
\end{tabular}
\caption{Estimates of GPD scale and shape parameters for $\nu=0.7$ and
$\sigma=3$ (the moment method has little sense since $\nu>0.5$) for the first
scenario without outliers}%
\label{numeric1}%
\end{table}

\begin{table}[ptb]
\center
\begin{tabular}
[c]{|c|c||c|c|c||c|c|c|}\cline{3-8}%
\multicolumn{2}{c||}{} & \multicolumn{3}{|c|}{$n=30$} &
\multicolumn{3}{|c|}{$n=100$}\\\hline
Estimation method & Parameter & Mean & Median & StD & Mean & Median &
StD\\\hline
$\chi^{2}$-divergence & $\sigma$ & 12.43 & 12.24 & 2.83 & 12.29 & 12.21 &
1.62\\\hline
$KL_{m}$-divergence & $\sigma$ & 24.01 & 19.36 & 49.38 & 27.30 & 20.99 &
48.75\\\hline
L-moment method & $\sigma$ & 22.27 & 20.83 & 5.69 & 21.68 & 21.03 &
3.09\\\hline
Moment method & $\sigma$ & 80.97 & 76.27 & 20.89 & 80.93 & 76.84 &
31.09\\\hline
MLE & $\sigma$ & 3.06 & 2.88 & 1.08 & 2.88 & 2.86 & 0.55\\\hline\hline
$\chi^{2}$-divergence & $\nu$ & 0.55 & 0.55 & 0.05 & 0.54 & 0.54 &
0.04\\\hline
$KL_{m}$-divergence & $\nu$ & 0.50 & 0.52 & 0.24 & 0.54 & 0.49 & 0.27\\\hline
L-moment method & $\nu$ & 0.54 & 0.54 & 0.06 & 0.54 & 0.53 & 0.04\\\hline
Moment method & $\nu$ & 0.07 & 0.08 & 0.02 & 0.08 & 0.07 & 0.03\\\hline
MLE & $\nu$ & 1.48 & 1.44 & 0.22 & 1.50 & 1.49 & 0.11\\\hline
\end{tabular}
\caption{Estimates of GPD scale and shape parameters for $\nu=0.7$ and
$\sigma=3$ for a sample with $10\%$ outliers of value $300$ (the moment method
has little meaning since $\nu>0.5$)}%
\label{numeric2}%
\end{table}

\begin{table}[ptb]
\center
\begin{tabular}
[c]{|c|c||c|c|c||c|c|c|}\cline{3-8}%
\multicolumn{2}{c||}{} & \multicolumn{3}{|c|}{$n=30$} &
\multicolumn{3}{|c|}{$n=100$}\\\hline
Estimation method & Parameter & Mean & Median & StD & Mean & Median &
StD\\\hline
$\chi^{2}$-divergence & $\sigma$ & 4.32 & 4.23 & 0.91 & 4.45 & 4.42 &
0.51\\\hline
$KL_{m}$-divergence & $\sigma$ & 5.04 & 4.90 & 1.15 & 5.07 & 5.08 &
0.67\\\hline
L-moment method & $\sigma$ & 5.18 & 5.04 & 1.44 & 5.11 & 5.04 & 0.75\\\hline
Moment method & $\sigma$ & 8.64 & 8.44 & 0.92 & 8.54 & 8.48 & 0.50\\\hline
MLE & $\sigma$ & 3.12 & 3.08 & 0.87 & 3.08 & 3.05 & 0.49\\\hline\hline
$\chi^{2}$-divergence & $\nu$ & 0.27 & 0.28 & 0.08 & 0.27 & 0.27 &
0.05\\\hline
$KL_{m}$-divergence & $\nu$ & 0.25 & 0.25 & 0.09 & 0.24 & 0.24 & 0.05\\\hline
L-moment method & $\nu$ & 0.24 & 0.24 & 0.10 & 0.24 & 0.24 & 0.06\\\hline
Moment method & $\nu$ & 0.01 & 0.02 & 0.04 & 0.01 & 0.02 & 0.02\\\hline
MLE & $\nu$ & 0.56 & 0.54 & 0.17 & 0.55 & 0.55 & 0.09\\\hline
\end{tabular}
\caption{Estimates of GPD scale and shape parameters for $\nu=0.1$ and
$\sigma=3$ for a sample with $10\%$ outliers of value $30$}%
\label{numeric3}%
\end{table}

\begin{table}[ptb]
\center
\begin{tabular}
[c]{|c||c|c|c|c||c|c|c|c|}\cline{2-9}%
\multicolumn{1}{c||}{} & \multicolumn{4}{|c|}{$n=30$} &
\multicolumn{4}{|c|}{$n=100$}\\\hline
Estimation method & Sc 1 & Sc 2 & Sc 3 & Sc 4 & Sc 1 & Sc 2 & Sc 3 & Sc
4\\\hline
$\chi^{2}$-divergence & 2.53 & 7.20 & 3.16 & 2.63 & 1.55 & 7.32 & 3.28 &
1.80\\\hline
L-moment method & 3.10 & 10.07 & 4.09 & 4.31 & 1.70 & 9.93 & 4.07 &
3.51\\\hline
Moment method & 6.79 & 14.47 & 7.07 & 8.69 & 6.91 & 14.42 & 6.98 &
9.98\\\hline
MLE & 1.78 & 2.83 & 2.68 & 11.69 & 0.97 & 2.42 & 2.33 & 9.25\\\hline
\end{tabular}
\caption{$L_{1}$-distances (to be multiplied by $10^{-4}$) between GPD
densities for different scenarios; Scenario (Sc) 1 corresponds to a simulated
GPD with $\nu=0.7$ and $\sigma=3$; Scenario 2 corresponds to a simulated GPD
with $\nu=0.7$, $\sigma=3$ and $10\%$ outliers of value $300$; Scenario 3
corresponds to a simulated GPD with $\nu=0.1$, $\sigma=3$ and $10\%$ outliers
of value $30$; Scenario 4 corresponds to a simulated Weibull distribution with
$\nu=0.4$ and $\sigma=3$}%
\end{table}

\begin{figure}[h!]
\centering
\subfloat[Simulated GPD with $\nu=0.7$ and $\sigma=3$]{\includegraphics[width=3in,height=3in]{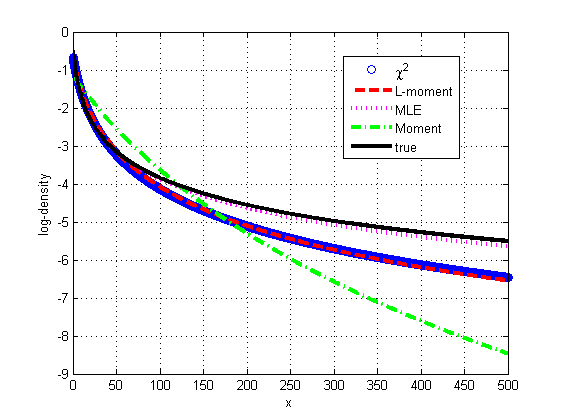}}
\subfloat[Simulated GPD with $\nu=0.7$, $\sigma=3$ and $10\%$ outliers of value $300$]{\includegraphics[width=3in,height=3in]{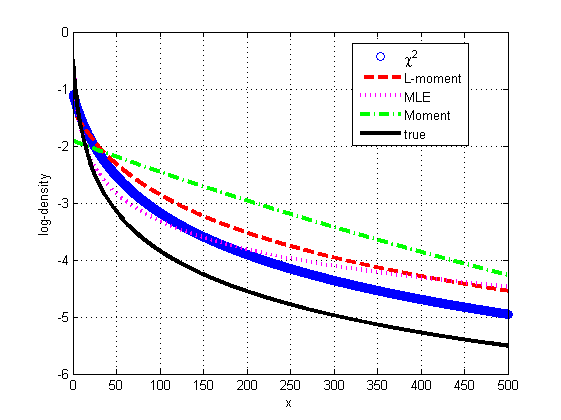}}\newline%
\subfloat[Simulated GPD with $\nu=0.1$, $\sigma=3$ and $10\%$ outliers of value $30$]{\includegraphics[width=3in,height=3in]{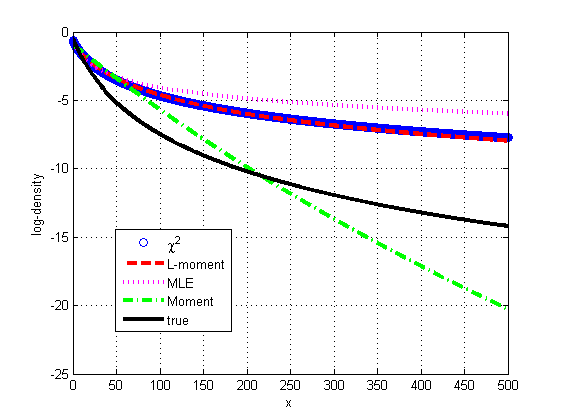}}
\subfloat[Simulated Weibull distribution with $\nu=0.4$ and $\sigma=3$]{\includegraphics[width=3in,height=3in]{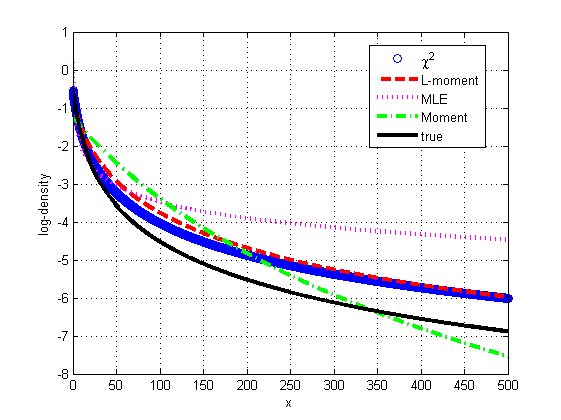}}\caption{Estimated
GPD densities with estimated parameters for simulated scenarios (with a
logarithmic scale)}%
\end{figure}

Unsurprisingly, the MLE performs well under the model and the L-moment method
has an overall better behavior than the classical moment method for the
considered heavy-tailed distributions (see Table \ref{numeric1}). Furthermore,
we observe that the $\chi^{2}$- divergence is more robust than the modified
Kullback as indeed expected.\newline The interesting result lies in their
behavior with outliers and misspecification. Indeed, we can see that
L-moment-based estimators perform well on the shape parameter whereas the MLE
provides a good estimation of the scale parameter but overestimates the shape
parameter. In that sense, the L-moments method can be used for the robust
estimation of the shape parameter of a GPD in case of contamination by
outliers. However, even with outliers, the MLE performs well in term of
$L_{1}$-distance computed on the estimated densities. It is under
misspecification that the performance of the MLE drops as measured by the
$L_{1}$ criterion. This confirms the flexibility of models defined only
through moment or L-moment equations that are less dependent on the GPD
model.\newline Moreover, the $L_{1}$-distance between the model and its
estimation has an order between $10^{-3}$ and $10^{-4}$. The error committed
by the estimation under models defined through L-moments conditions is the
most stable over the proposed scenarios. We can then affirm that we can
estimate the probability of events if the true value of this probability is of
order $10^{-3}$ (the error of estimation for the estimator based on L-moments
method would approximately be of $30\%$ depending on the size of the sample
and the scenario).

\appendix

\section{Proofs}

\subsection{Proof of Lemma \ref{transport1}}

Let $x\in\mathbb{R}$. We denote by $F$ the cdf of $X$ and by $A_{t}$ the
event
\[
A_{t}=\left\{  x\in\mathbb{R}\text{ s.t. }F(x)\geq t\right\}
\]
We then have $Q(t)=\inf A_{t}$. We wish to prove :
\begin{equation}
\left\{  t\in\lbrack0;1]\text{ s.t. }Q(t)\leq x\right\}  =\left\{  t\in
\lbrack0;1]\text{ s.t. }t\leq F(x)\right\}  \label{equivalence}%
\end{equation}
We temporarily admit this assertion. Then%
\[
\mathbb{P}[Q(U)\leq x]=\mathbb{P}[U\leq F(x)]=F(x)
\]
which ends the proof. It remains to prove (\ref{equivalence}).\newline First,
the definition of $Q$ yields
\[
\left\{  t\leq F(x)\right\}  \Rightarrow\left\{  x\in A_{t}\right\}
\Rightarrow\left\{  Q(t)\leq x\right\}  .
\]
Secondly, let $t$ be such that $Q(t)\leq x$. Then by monotonicity of $F$,
$F(Q(t))\leq F(x)$. We then claim that
\[
Q(t)\in A_{t}.
\]
Indeed, let us suppose the contrary and consider a strictly decreasing
sequence $x_{n}\in A_{t}$ such that
\[
\lim_{n\rightarrow\infty}x_{n}=\inf A_{t}=Q(t).
\]
By right continuity of $F$
\[
\lim_{n\rightarrow\infty}F(x_{n})=F(Q(t))
\]
and, on the other hand, by definition of $A_{t}$,
\[
\lim_{n\rightarrow\infty}F(x_{n})\geq t
\]
i.e. $Q(t)\in A_{t}$ which contradicts the hypothesis. Then $Q(t)\in A_{t}$
i.e. $t\leq F(Q(t))$ thus $t\leq F(x)$. We have proved that
\[
\left\{  Q(t)\leq x\right\}  \Rightarrow\left\{  t\leq F(x)\right\}  .
\]

\subsection{Proof of Lemma \ref{stieljes}}

Let us recall that the support of a measure $\mu$ defined on $X\subset
\mathbb{R}$ is the largest closed set $C\subset X$ such that
\[
U\in B(X)\text{ and }U\cap C\neq\emptyset\Rightarrow\mu(U\cap C)>0
\]
where\ $\ B(X)$ denotes the Borel sets in $X$. Let $S$ be the support of
$\mathbf{F}^{-1}$. Then $[0;1]\backslash S$ is an open set in $[0;1]$ i.e. a
countable union of intervals $\cup_{i\geq1}]t_{2i},t_{2i+1}[$ and
\begin{align*}
\int_{0}^{1}a(F^{-1}(t))dt  &  =\int_{S}a(F^{-1}(t))dt+\sum_{i\geq1}%
\int_{]t_{2i};t_{2i+1}[}a(F^{-1}(t))dt\\
&  =\int_{F^{-1}(S)}a(x)dF(x)+\sum_{i\geq1}a(F^{-1}(t_{2i}))(t_{2i+1}%
-t_{2i})\\
&  =\int_{F^{-1}(S)}a(x)dF(x)+\sum_{i\geq1}\int_{\{F^{-1}(t_{2i}%
)\}}a(x)dF(x)\\
&  =\int_{F^{-1}(S)\cup\left(  \cup_{i\geq1}\{F^{-1}(t_{2i})\}\right)
}a(x)dF(x).
\end{align*}
The second equality stems from the definition of the quantile as
left-continuous function and from the fact that $F^{-1}$ is strictly monotone
on $S$.\newline As $F^{-1}$ is constant on the open interval $[t_{2i}%
;t_{2i+1}[$, $\{F^{-1}(t_{2i})\}=F^{-1}([t_{2i};t_{2i+1}[)$. Hence
\begin{align*}
F^{-1}(S)\cup\left(  \cup_{i\geq1}\{F^{-1}(t_{2i})\}\right)   &
=F^{-1}([0;1])\\
&  =\{x\in\mathbb{R}\text{ s.t. there exists $t$ with $F^{-1}(t)=x$%
}\}=supp(F).
\end{align*}
We conclude the first part of the proof since
\[
\int_{supp(F)}a(x)dF(x)=\int_{\mathbb{R}}a(x)dF(x).
\]

The second part of the proof can be proved similarly since the above arguments
are not particular to a specific measure.

\subsection{Proof of Proposition \ref{prop2}}

The proof is directly adapted from the proof of Theorem II.2 of Csisz\'ar et
al. \cite{csiszar99}.\newline Let us begin with the fundamental lemma inspired
from Theorem 2.9 of Borwein and Lewis\cite{borwein91}.\newline

\begin{lemma}
\label{LemmaRestriction}Let $C:\Omega\rightarrow\mathbb{R}^{l}$ be an array of
bounded functions such that
\[
\int_{\Omega}\Vert C(x)\Vert d\mu(x)<\infty.
\]
We denote
\[
L_{C,a}=\left\{  g\text{ s.t. }\int_{\Omega}g(t)C(t)d\mu(t)=a\right\}  .
\]
If there exists some $g$ in $L_{C,a}$ such that $a_{\varphi}<g<b_{\varphi}$
$\mu$-a.s and $\int_{\Omega}\Vert g(t)C(t)\Vert d\mu(t)<\infty$, then there
exists $a_{\varphi}^{\prime}>a_{\varphi}$, $b_{\varphi}^{\prime}<b_{\varphi}$
and $g_{b}\in L_{C,a}$ such that $a_{\varphi}^{\prime}\leq g_{b}(x)\leq
b_{\varphi}^{\prime}$ for all $x\in\Omega$.
\end{lemma}

\begin{proof}
Let $L$ denotes the subspace of $\mathbb{R}^{l}$ composed by the vectors
representable as $\int_{\Omega}gCd\mu$ for some $g:\Omega\rightarrow
\mathbb{R}^{l}$. Let us denote by $a_{n}$ a decreasing sequence $a_{n}%
\rightarrow a_{\varphi}$, by $b_{n}$ a increasing one $b_{n}\rightarrow
b_{\varphi}$ and let $T_{n}$ be the set
\[
T_{n}=\left\{  x\in\Omega\text{ s.t. }a_{n}\leq g(x)\leq b_{n}\right\}  .
\]
We first claim that, for $n$ large enough
\[
L=L_{n}=\left\{  \int_{\Omega}hCd\mu\text{ with $h(x)=0$ if $x\not \in T_{n}$
and $h$ bounded}\right\}  .
\]
Indeed, if not, we can build a sequence of vectors $v_{n}$ such that $\Vert
v_{n}\Vert=1$, $v_{n}\in L^{\perp}$ and $v_{n}\rightarrow v\in L$.
Furthermore, $v_{n}\in L^{\perp}$ means
\[
\langle v_{n},\int_{\Omega}hCd\mu\rangle=\int_{\Omega}h\langle v_{n},C\rangle
d\mu=0
\]
then $\langle v_{n},C\rangle=0$ for all $x\in T_{n}$ $\mu$-a.s. Hence $\langle
v,C\rangle=0$ $\mu$-a.s. and $v\in L^{\perp}$ which contradicts $v\in L$ with
$\Vert v\Vert=1$.\newline Let us then fix some $n_{0}$ such that $L_{n_{0}}%
=L$. We denote by
\[
L_{n}(\delta)=\left\{  \int_{\Omega}hCd\mu\text{ with $h(x)=0$ if
$x\not \in T_{n}$ and $|h(x)|<\delta$ for $x\in\Omega$}\right\}  .
\]
Then, the affine hull of $L_{n}(\delta)$ is the vector space $L$ and $0\in
L_{n}(\delta)$. We can consider the function $g_{n}$
\[
g_{n}(x)=\left\{
\begin{array}
[c]{lll}%
a_{n} & \text{ if } & g(x)<a_{n}\\
g(x) & \text{ if } & b_{n}\leq g(x)\leq a_{n}\\
b_{n} & \text{ if } & g(x)>b_{n}%
\end{array}
\right.
\]
Then $\Vert\int_{\Omega}(g_{n}-g)Cd\mu\Vert\rightarrow_{n\rightarrow\infty}0$.
Indeed we can apply the dominated convergence theorem since, for any
$x\in\Omega$, $g_{n}(x)\rightarrow g$ and
\begin{align*}
\Vert(g_{n}(x)-g(x))C(x)\Vert &  =\Vert\mathds{1}_{g(x)<a_{n}}(a_{n}%
-g(x))C(x)+\mathds{1}_{g(x)>b_{n}}(g(x)-b_{n})C(x)\Vert\\
&  \leq(\Vert a_{0}-g(x)\Vert+\Vert b_{0}-g(x)\Vert)\Vert C(x)\Vert\\
&  \leq(\Vert a_{0}\Vert+\Vert b_{0}\Vert)\Vert C(x)\Vert+2\Vert
g(x)\Vert\Vert C(x)\Vert
\end{align*}
which is $\mu$-measurable by hypothesis.\newline We conclude that
$\int_{\Omega}(g_{n}-g)Cd\mu\in L_{n_{0}}(\delta)$ for $n$ large enough
because $0\in L_{n_{0}}(\delta)$. Hence there exists $h$ such that
$\int_{\Omega}(g_{n}-g)Cd\mu=\int_{\Omega}hCd\mu$, $|h(x)|=0$ for
$x\not \in T_{n_{0}}$ and $|h(x)|<\delta$ for $x$ in $T_{n_{0}}$.\newline
Therefore for $x\in\Omega$, $\min(a_{n},a_{n_{0}}-\delta)\leq g_{n}%
(x)+h(x)\leq\min(b_{n},b_{n_{0}}+\delta)$ and $\int_{\Omega}(g_{n}%
+h)Cd\mu=\int_{\Omega}gCd\mu$. As $\delta$ is arbitrarily small, $h$ is the
null function.
\end{proof}

We can now prove the duality equality. Let note for $c\in\mathbb{R}^{l}$
$I(c)=\inf_{\int gCd\mu=c}\int_{\Omega}\varphi(g)d\mu$ and
\[
J(c)=\left\{
\begin{array}
[c]{ll}%
0 & \text{ if $c=a$}\\
+\infty & \text{ otherwise}%
\end{array}
\right.  .
\]
Then
\[
\inf_{g\in L_{C,a}}\int\varphi(g)d\mu=\inf_{c\in\mathbb{R}^{l}}I(c)+J(c).
\]
Recall that the Fenchel duality theorem (\cite{rockafellar70} p327) states
that if $ri(dom(I))\cap ri(dom(J))\neq\emptyset$ then
\[
\inf_{c\in\mathbb{R}^{l}}I(c)+J(c)=\max_{\xi\in\mathbb{R}^{l}}-I^{\ast}%
(\xi)-J^{\ast}(-\xi).
\]
We prove that $ri(dom(I))\cap ri(dom(J))\neq\emptyset$. Note that
$ri(dom(J))=\{a\}$. It suffices then to prove that $a$ belongs to
$int(dom(I))$ for the topology induced by $L$. By the above Lemma
\ref{LemmaRestriction} there exists $g_{b}$ such that $a_{\varphi}<a_{\varphi
}^{\prime}\leq g_{b}(x)\leq b_{\varphi}^{\prime}<b_{\varphi}$ for all
$x\in\Omega$. Since $a+L_{n}(\delta)$ is a neighborhood of $a$ included in
$dom(I)$ for $\delta$ sufficiently small, it holds that $a\in int(L_{n}%
(\delta))\subset int(dom(I))$.\newline It remains now to compute the
conjugates of $I$ and $J$ .
\begin{align*}
I^{\ast}(\xi)  &  =\sup_{c\in\mathbb{R}^{l}}\langle\xi,c\rangle-\inf_{g,\int
gCd\mu=c}\varphi(g)d\mu\\
&  =\sup_{c\in\mathbb{R}^{l}}\sup_{g,\int gCd\mu=c}\langle\xi,c\rangle
-\varphi(g)d\mu\\
&  =\sup_{g}\langle\xi,\int gCd\mu\rangle-\varphi(g)d\mu\\
&  =\sup_{g}\int\langle\xi,C\rangle g-\varphi(g)d\mu\\
&  =\int\psi(\langle\xi,C\rangle)d\mu
\end{align*}
This equality is referred to as the integral representation of $I^{\ast}$. The
last equality can be rigorously justified (see for example \cite{matus12}%
).\newline Furthermore, $J^{\ast}(-\xi)=-\langle\xi,a\rangle$ which closes the
first part of the proof, namely
\[
\inf_{g\in L_{C,a}}\int_{\Omega}\varphi\left(  g\right)  d\mu=\sup_{\xi
\in\mathbb{R}^{l}}\langle\xi,a\rangle-\int_{\Omega}\psi(\langle\xi
,C(x)\rangle)d\mu.
\]
As we assume $\psi$ differentiable, then $\xi\mapsto\langle\xi,a\rangle
-\int_{\Omega}\psi(\langle\xi,C(x)\rangle)d\mu$ is differentiable as well. It
follows that any critical point is the solution of
\[
\int_{\Omega}\psi^{\prime}(\langle\xi,C(x)\rangle)C(x)d\mu=a.
\]
Furthermore, as $\varphi$ is strictly convex, $\psi$ is strictly concave and
for $\xi,\xi^{\prime}\in\mathbb{R}^{l}$ and $t\in\lbrack0;1]$ it holds
\begin{align*}
&  \langle(1-t)\xi+t\xi^{\prime},a\rangle-\int_{\Omega}\psi(\langle
(1-t)\xi+t\xi^{\prime},C(x)\rangle)d\mu\\
&  =\langle(1-t)\xi+t\xi^{\prime},a\rangle-\int_{\Omega}\psi((1-t)\langle
\xi,C(x)\rangle+t\langle\xi^{\prime},C(x)\rangle)d\mu\\
&  <(1-t)\left[  \langle\xi,a\rangle-\int_{\Omega}\psi(\xi,C(x)\rangle
)d\mu\right]  +t\left[  \langle\xi^{\prime},a\rangle-\int_{\Omega}\psi
(\xi^{\prime},C(x)\rangle)d\mu\right]
\end{align*}
i.e. the functional $\xi\rightarrow\langle\xi,a\rangle-\int_{\Omega}\psi
(\xi,C(x)\rangle)d\mu$ is strictly convex which proves the uniqueness of
$\xi^{\ast}$.\newline The continuity of $a\mapsto\xi^{\ast}(a)$ comes from the
implicit function theorem. If we note $D(\xi)=\int\psi^{\prime}(\langle
\xi,C(x)\rangle)C(x)d\mu$ then $D$ is continuously differentiable with a
Jacobian given by
\[
J_{D}(\xi)=\int\psi^{\prime\prime} (\langle\xi,C(x)\rangle)C(x)C(x)^{T}d\mu
\]
which is positive definite thanks to the strict convexity of $\psi$.

\subsection{Proof of Proposition \ref{prop40}}

Note
\[
D_{\varphi}(x,y) = \sum_{i=1}^{n-1}\varphi\left(  \frac{y_{i+1}-y_{i}%
}{x_{i+1:n}-x_{i:n}}\right)  (x_{i+1:n}-x_{i:n}).
\]
Assuming that (\ref{Prop1}) holds then (\ref{Prop2}) follows from equation
(\ref{lmom_est}). Indeed, since $\varphi$ is infinite for negative values, it
holds
\begin{align*}
& \inf_{\substack{y\in\mathbb{R}^{n}\\\sum_{i=1}^{n-1}K(i/n)(y_{i+1:n}%
-y_{i:n})=f(\theta)}}D_{\varphi}(x,y)\\
& =\inf_{\substack{(y_{1}<...<y_{n})\in\mathbb{R}^{n}\\\sum_{i=1}%
^{n-1}K(i/n)(y_{i+1}-y_{i})=f(\theta)}}D_{\varphi}(x,y).
\end{align*}
We now turn to (\ref{Prop1}). The minimization problem can be decomposed into
\begin{align*}
&\inf_{T\in\mathcal{C}_{A}:T\#\mathbf{F}_{n}\in L_{\theta}}\int_{\mathbb{R}%
}\varphi\left(  \frac{dT}{d\lambda}\right)  d\lambda\nonumber\\
=&\inf_{\substack{(y_{1},...,y_{n})\in\mathbb{R}^{n}\\\sum_{i=1}^{n}K(i/n)(y_{i+1:n}-y_{i:n}%
)=f(\theta)}}I_{\varphi}(x,y)
\end{align*}
by denoting
\[
I_{\varphi}(x,y) = \inf_{T\in\mathcal{C}_{A}:T(x_{i:n})=y_{i}}\int
_{\mathbb{R}}\varphi\left(  \frac{dT}{d\lambda}\right)  d\lambda
\]
This minimization problem has an explicit solution.
Indeed
\[
I_{\varphi}(x,y)=\inf_{T\in\mathcal{C}_{A}:T(x_{i+1:n}%
)-T(x_{i:n})=y_{i+1}-y_{i}}\int_{\mathbb{R}}\varphi\left(  \frac{dT}{d\lambda
}\right)  d\lambda.
\]
If $T\in\mathcal{C}_{A}$ satisfies $T(x_{i:n})=y_{i}$ for $1\leq i\leq n$
$\ $then  as $T$ is absolutely continuous, it holds for all $1\leq i\leq n-1$
\[
\int_{x_{i:n}}^{x_{i+1:n}}\frac{dT}{d\lambda}d\lambda=y_{i+1}-y_{i}.
\]
Conversely, if $S:\mathbb{R}\rightarrow\mathbb{R}$ is such that for all $i$
between $1$ and $n-1$
\[
\int_{x_{i:n}}^{x_{i+1:n}}S(x)\lambda(dx)=y_{i+1}-y_{i}%
\]
then $T:x\mapsto\int_{0}^{x}S(x)\lambda(dx)\in\mathcal{C}_{A}$ and
$T(x_{i+1:n})-T(x_{i:n})=y_{i+1}-y_{i}$.\newline We thus obtain
\[
I_{\varphi}(x,y)=\inf_{\substack{S:\int_{\mathbb{R}%
}S(x)\mathds{1}_{\{x_{i:n}\leq x\leq x_{i+1:n}\}}\lambda(dx)\\=y_{i+1}%
-y_{i},1\leq i\leq n-1}}\int_{\mathbb{R}}\varphi\left(  S(x)\right)
\lambda(dx)
\]
From Proposition \ref{prop2}, it then holds, since $\psi(0)=0$
\begin{align*}
& \inf_{S:\int_{\mathbb{R}}S(x)\mathds{1}_{\{x_{i:n}\leq x\leq x_{i+1:n}%
\}}\lambda(dx)=y_{i+1}-y_{i}}\int_{\mathbb{R}}\varphi\left(  S(x)\right)
\lambda(dx)\\
& =\sup_{(\xi_{1},...,\xi_{n-1})\in\mathbb{R}^{n-1}}\sum_{i=1}^{n-1}\xi
_{i}(y_{i+1}-y_{i})-\psi(\xi_{i})(x_{i+1:n}-x_{i:n})\\
& =\sum_{i=1}^{n-1}\sup_{\xi_{i}\in\mathbb{R}}\xi_{i}(y_{i+1}-y_{i})-\psi
(\xi_{i})(x_{i+1:n}-x_{i:n})\\
& =\sum_{i=1}^{n-1}(x_{i+1:n}-x_{i:n})\sup_{\xi_{i}\in\mathbb{R}}\xi_{i}%
\frac{y_{i+1}-y_{i}}{x_{i+1:n}-x_{i:n}}-\psi(\xi_{i})\\
& =\sum_{i=1}^{n-1}(x_{i+1:n}-x_{i:n})\varphi\left(  \frac{y_{i+1}-y_{i}%
}{x_{i+1:n}-x_{i:n}}\right)
\end{align*}
which concludes the proof.

\subsection{Proof of Theorem \ref{th1}}

The arguments of this proof and of the following one are similar to the ones
given by Newey and Smith in \cite{neweysmith04} for their Theorem 3.1; the
essential argument is a Taylor expansion of the functionals in equation
(\ref{eq:511}).\newline Let begin with a lemma adapted from Theorem 6 due to
Stigler \cite{stigler74} :

\begin{lemma}
\label{lemma_asym} Let $x_{1},...,x_{n}$ be an observed sample drawn iid from
a distribution $F$ with finite variance. We note $F_{n}$ the empirical
distribution of the sample. \newline Let $A:[0;1]\rightarrow\mathbb{R}^{l}$ be
a continuously derivable function such that $A^{\prime}$ is bounded $F^{-1}%
$-a.e. Then
\[
n^{1/2}\left(  \int xdA(F_{n}(x))-\int xdA(F(x))\right)  \rightarrow
_{d}N(0,\Sigma_{A})
\]
with
\[
\Sigma_{A}=\iint\left[  F(\min(x,y))-F(x)F(y)\right]  A^{\prime}%
(F(x))A^{\prime}(F(y))^{T}dxdy.
\]

\end{lemma}

In the following, we will note $\frac{dT}{d\lambda}(x)=T^{\prime}(x)$ for all
$x\in\mathbb{R}$ .\newline

\textbf{First step} : maximization step\newline

Clearly, it holds
\begin{equation}
\inf_{T\in\cup_{\theta}L_{\theta}^{\prime\prime}(F_{n})}\int_{\mathbb{R}%
}\varphi(T^{\prime}(x))dx\leq\inf_{T\in L_{\theta_{0}}^{\prime\prime}(F_{n}%
)}\int_{\mathbb{R}}\varphi(T^{\prime}(x))dx. \label{eq_proof1}%
\end{equation}
By Taylor-Lagrange expansion, there exists some $D>0$ such that for $n$ large
enough and for any $t$ in $[1-n^{-1/4};1+n^{1/4}]$
\[
\varphi(t)\leq\frac{D}{2}(t-1)^{2}%
\]
holds.

We may then majorize the RHS in (\ref{eq_proof1}) by the solution of the
quadratic case. Let
\[
T_{0,n}^{\prime}(x):=1+(f(\theta_{0})-m_{n})^{T}\Omega_{n}^{-1}K(F_{n}(x))
\]
where $m_{n}:=\int K(F_{n}(x))dx$ and $\Omega_{n}:=\int_{\mathbb{R}}%
K(F_{n}(x))K(F_{n}(x))^{T}dx$. As $T_{0,n}^{\prime} \in L_{\theta}%
^{\prime\prime}(F_{n}) $, it holds
\[
\inf_{T\in L_{\theta_{0}}^{\prime\prime}(F_{n})}\int_{\mathbb{R}}%
\varphi(T^{\prime}(x))dx\leq\int_{\mathbb{R}}\varphi(T_{0,n}^{\prime}(x))dx.
\]

From Lemma \ref{lemma_asym}, we deduce that $\Omega_{n}\rightarrow\Omega$ in
probability. As $\Omega$ is non singular, for $n$ large enough, $\Omega_{n}$
is non singular and $T_{0,n}^{\prime}$ is well defined.\newline As $\Vert
f(\theta_{0})-m_{n}\Vert=O_{P}(n^{-1/2})$ from Lemma \ref{lemma_asym} and
$\Vert\Omega_{n}^{-1}\Vert=O_{P}(1)$, for almost all $x\in\mathbb{R}$,
\[
T_{0,n}^{\prime}(x)=1+O_{P}(n^{-1/2})
\]
and we can apply a Taylor-Lagrange maximization
\[
\varphi(T_{0,n}^{\prime}(x))\leq\frac{D}{2}(f(\theta_{0})-m_{n})^{T}\Omega
_{n}^{-1}K(F_{n}(x))K(F_{n}(x))^{T}\Omega_{n}^{-1}(f(\theta_{0})-m_{n}).
\]
By integration in the above display
\begin{align*}
\int_{\mathbb{R}}\varphi(T_{0,n}^{\prime}(x))dx  &  \leq\frac{D}{2}%
(f(\theta_{0})-m_{n})\Omega_{n}^{-1}\left[  \int_{\mathbb{R}}K(F_{n}%
(x))K(F_{n}(x))^{T}dx\right]  \Omega_{n}^{-1}(f(\theta_{0})-m_{n})\\
&  \leq\Vert f(\theta_{0})-m_{n}\Vert^{2}\Vert\Omega_{n}^{-1}\Vert
=O_{P}(n^{-1}).
\end{align*}

\textbf{Second step} : minimization step\newline

Since $\Theta$ is compact, and $\varphi$ is strictly convex, and
$\theta\mapsto\inf_{T\in L_{\theta}^{\prime\prime}(F_{n})}\int_{\mathbb{R}%
}\varphi(T^{\prime}(x))dx$ is continuous (see Proposition \ref{prop2}), it
follows that $\hat{\theta}$ is well defined and the duality equality states
\begin{align*}
\inf_{T\in L_{\hat{\theta}_{n}}^{\prime\prime}(F_{n})}\int_{\mathbb{R}}%
\varphi(T^{\prime}(x))dx  &  =\sup_{\xi\in\mathbb{R}^{l}}\xi^{T}f(\hat{\theta
}_{n})-\int\psi(\xi^{T}K(F_{n}(x)))dx\\
&  \geq\xi_{n}^{T}f(\hat{\theta}_{n})-\int\psi(\xi_{n}^{T}K(F_{n}(x)))dx
\end{align*}
with
\[
\xi_{n}=n^{-1/2}\frac{f(\hat{\theta}_{n})-m_{n}}{\Vert f(\hat{\theta}%
_{n})-m_{n}\Vert}.
\]
.\newline Therefore%
\[
\xi_{n}^{T}K(F_{n}(x))=O_{P}(n^{-1/2})\text{ for a.e $x\in\mathbb{R}$}.
\]
By Taylor-Lagrange expansion, there exists a constant $C>0$ such that
$|\psi(x)-x|<Cx^{2}$ in a neighborhood of 0. Thus, for $n$ large enough
\[
\int\psi(\xi_{n}^{T}K(F_{n}(x)))dx-\xi_{n}^{T}m_{n}<C\int\xi_{n}^{T}%
K(F_{n}(x))K(F_{n}(x))^{T}\xi_{n}dx=C\xi_{n}^{T}\Omega_{n}\xi_{n}%
\]
and
\[
\inf_{T\in L_{\hat{\theta}_{n}}^{\prime\prime}(F_{n})}\int_{\mathbb{R}}%
\varphi(T^{\prime}(x))dx\newline>\xi_{n}^{T}(f(\hat{\theta}_{n})-m_{n}%
)-C\xi_{n}^{T}\Omega_{n}\xi_{n}.
\]

\textbf{Conclusion}\newline Combining the two inequalities, we have
\[
n^{-1/2}\Vert f(\hat{\theta}_{n})-m_{n}\Vert<C\Vert\Omega_{n}\Vert
n^{-1}+\Vert f(\theta_{0})-m_{n}\Vert^{2}\Vert\Omega_{n}^{-1}\Vert
=O_{P}(n^{-1})
\]
i.e. $\Vert f(\hat{\theta}_{n})-m_{n}\Vert=O_{P}(n^{-1/2})$.\newline By Lemma
\ref{lemma_asym}, $\Vert m_{n}-f(\theta_{0})\Vert=O_{P}(n^{-1/2})$. Hence,
$\Vert f(\hat{\theta}_{n})-f(\theta_{0})\Vert=O_{P}(n^{-1/2})$.\newline Since
$f(\theta)=f(\theta_{0})$ has a unique solution at $\theta_{0}$, $\Vert
f(\theta)-f(\theta_{0})\Vert$ is bounded away from zero outside some
neighborhood of $\theta_{0}$. Therefore $\hat{\theta}_{n}$ is inside any
neighborhood of $\theta_{0}$ with probability approaching 1 i.e $\hat{\theta
}_{n}\rightarrow\theta_{0}$ in probability.

\subsection{Proof of Theorem \ref{theorem_asymp}}

First we prove that
\[
\hat{\xi}_{n}=\arg\max_{\xi}\xi^{T}f(\hat{\theta}_{n})-\int\psi(\xi^{T}%
K(F_{n}(x))dx=O_{P}(n^{-1/2}).
\]
Consider
\[
\xi_{n}=\arg\max_{\xi\in\mathbb{R}^{l}\text{ s.t. }\Vert\xi\Vert<n^{-1/4}}%
\xi^{T}f(\hat{\theta}_{n})-\int\psi(\xi^{T}K(F_{n}(x))dx,
\]
where the maximum is taken on a ball of radius $n^{-1/4}$. The maximum is
attained because of the concavity of the functional
\[
U:\xi\mapsto\xi^{T}f(\hat{\theta}_{n})-\int\psi(\xi^{T}K(F_{n}(x))dx.
\]
For all $x$ in a neighborhood of $0$, the inequality $y-\psi(y)<-Cy^{2}$ for
some $C>0$ holds. For $n$ large enough, as $\Vert\xi_{n}\Vert<n^{-1/4}$ we can
claim (as $\psi(0)=0$)
\begin{align*}
0  &  \leq\xi_{n}^{T}f(\hat{\theta}_{n})-\int\psi(\xi_{n}^{T}K(F_{n}(x))dx\\
&  \leq\xi_{n}^{T}(f(\hat{\theta}_{n})-m_{n})-C\xi_{n}^{T}\Omega_{n}\xi_{n}\\
&  \leq\Vert\xi_{n}\Vert.\Vert f(\hat{\theta}_{n})-m_{n}\Vert-C\xi_{n}%
\Omega_{n}\xi_{n},
\end{align*}
with $m_{n}:=\int K(F_{n}(x))dx$.\newline Furthermore, there exists $D>0$ such
that $\Vert\Omega_{n}\Vert\geq D>0$ for $n$ large enough and
\[
CD\leq C\frac{\xi_{n}^{T}}{\Vert\xi_{n}\Vert}\Omega_{n}\frac{\xi_{n}}{\Vert
\xi_{n}\Vert}\leq\frac{\Vert f(\hat{\theta}_{n})-m_{n}\Vert}{\Vert\xi_{n}%
\Vert}.
\]
It follows that $\xi_{n}=O_{P}(n^{-1/2})$ and that $\xi_{n}$ is an interior
point of $\{\xi\in\mathbb{R}^{l}\text{ s.t. }\Vert\xi\Vert<n^{-1/4}\}$; by
concavity of the functional $U$, $\xi_{n}$ is the unique maximizer, hence
$\xi_{n}=\hat{\xi}_{n}$.\newline We write the first order conditions of
optimality of $(\hat{\theta}_{n}-\theta_{0},\hat{\xi}_{n})$ :
\[
\left\{
\begin{array}
[c]{l}%
(f(\hat{\theta}_{n})-f(\theta_{0}))+(f(\theta_{0})-m_{n})-\int\left[
\psi^{\prime}(\hat{\xi}_{n}K(F_{n}(x))-1\right]  K(F_{n}(x))dx=0\\
J_{f}(\hat{\theta}_{n})\hat{\xi}_{n}=0
\end{array}
\right.
\]
A mean value expansion (since $\theta_{0}\in int(\Theta)$) gives the existence
of $\bar{\xi}$ and $\bar{\theta}$ such that $\Vert\bar{\xi}\Vert<\Vert\hat
{\xi}_{n}\Vert$ and $\Vert\bar{\theta}-\theta_{0}\Vert<\Vert\hat{\theta}%
_{n}-\theta_{0}\Vert$ such that
\[
\left\{
\begin{array}
[c]{l}%
J_{f}(\bar{\theta})(\theta-\theta_{0})+(f(\theta_{0})-m_{n})-\left[  \int
\psi^{\prime\prime}(\bar{\xi}K(F_{n}(x))K(F_{n}(x))K(F_{n}(x))dx\right]
\hat{\xi}_{n}=0\\
J_{f}(\hat{\theta}_{n})\hat{\xi}_{n}=0
\end{array}
\right.  .
\]

It holds
\[
A_{n}:=\left(
\begin{array}
[c]{cc}%
J_{f}(\bar{\theta}) & -\int\psi^{\prime\prime}(\bar{\xi}K(F_{n}(x))K(F_{n}%
(x))K(F_{n}(x))dx\\
0 & J_{f}(\hat{\theta}_{n})
\end{array}
\right)  \rightarrow_{p}A:=\left(
\begin{array}
[c]{cc}%
J_{0} & -\Omega\\
0 & J_{0}%
\end{array}
\right)  .
\]
By the very definition of $A_{n}$,
\[
A_{n}\left(
\begin{array}
[c]{cc}%
\hat{\theta}_{n}-\theta_{0} & \\
\hat{\xi}_{n} &
\end{array}
\right)  =\left(
\begin{array}
[c]{cc}%
m_{n}-f(\theta_{0}) & \\
0 &
\end{array}
\right)  .
\]
As $\Omega$ is non singular and $J_{0}$ has full rank, $A$ is non singular and
its inverse is given by
\[
A^{-1}=\left(
\begin{array}
[c]{cc}%
H & M\\
P & H-H^{T}%
\end{array}
\right)  .
\]
Hence by Lemma \ref{lemma_asym}
\[
\sqrt{n}\left(
\begin{array}
[c]{cc}%
\hat{\theta}_{n}-\theta_{0} & \\
\hat{\xi}_{n} &
\end{array}
\right)  =A_{n}^{-1}\left(
\begin{array}
[c]{cc}%
\sqrt{n}(m_{n}-f(\theta_{0})) & \\
0 &
\end{array}
\right)  \rightarrow_{d}A^{-1}\left(
\begin{array}
[c]{cc}%
\mathcal{N}_{l}(0,\Sigma) & \\
0 &
\end{array}
\right)  ,
\]
which ends the proof.


\end{document}